\documentclass[11pt]{article}
\usepackage[centertags]{amsmath}
\usepackage{amsfonts}
\usepackage{amssymb}
\usepackage{amsthm,color}
\usepackage{newlfont}
\usepackage{bbm}

\newtheorem{Theorem}{Theorem}[section]

\newtheorem{Lemma}[Theorem]{Lemma}

\newtheorem{Remark}[Theorem]{Remark}

\newtheorem{Hypothesis}{Hypothesis}

\numberwithin{equation}{section}

\begin{document}

\def\le{\left}
\def\r{\right}
\def\cost{\mbox{const}}
\def\a{\alpha}
\def\d{\delta}
\def\ph{\varphi}
\def\e{\epsilon}
\def\la{\lambda}
\def\si{\sigma}
\def\La{\Lambda}
\def\B{{\cal B}}
\def\A{{\mathcal A}}
\def\L{{\mathcal L}}
\def\O{{\mathcal O}}
\def\bO{\overline{{\mathcal O}}}
\def\F{{\mathcal F}}
\def\K{{\mathcal K}}
\def\H{{\mathcal H}}
\def\D{{\mathcal D}}
\def\C{{\mathcal C}}
\def\M{{\mathcal M}}
\def\N{{\mathcal N}}
\def\G{{\mathcal G}}
\def\T{{\mathcal T}}
\def\R{{\mathbb R}}
\def\I{{\mathcal I}}

\def\bw{\overline{W}}
\def\phin{\|\varphi\|_{0}}
\def\s0t{\sup_{t \in [0,T]}}
\def\lt{\lim_{t\rightarrow 0}}
\def\iot{\int_{0}^{t}}
\def\ioi{\int_0^{+\infty}}
\def\ds{\displaystyle}
\def\pag{\vfill\eject}
\def\fine{\par\vfill\supereject\end}
\def\acapo{\hfill\break}

\def\beq{\begin{equation}}
\def\eeq{\end{equation}}
\def\barr{\begin{array}}
\def\earr{\end{array}}
\def\vs{\vspace{.1mm}   \\}
\def\rd{\reals\,^{d}}
\def\rn{\reals\,^{n}}
\def\rr{\reals\,^{r}}
\def\bD{\overline{{\mathcal D}}}
\newcommand{\dimo}{\hfill \break {\bf Proof - }}
\newcommand{\nat}{\mathbb N}
\newcommand{\E}{\mathbb E}
\newcommand{\Pro}{\mathbb P}
\newcommand{\com}{{\scriptstyle \circ}}
\newcommand{\reals}{\mathbb R}

\newcommand{\red}[1]{\textcolor{red}{#1}}

\def\Amu{{A_\mu}}
\def\Qmu{{Q_\mu}}
\def\Smu{{S_\mu}}
\def\H{{\mathcal{H}}}
\def\Im{{\textnormal{Im }}}
\def\Tr{{\textnormal{Tr}}}
\def\E{{\mathbb{E}}}
\def\P{{\mathbb{P}}}
\def\span{{\textnormal{span}}}
\title{On the convergence of stationary solutions in the Smoluchowski-Kramers approximation of infinite dimensional systems}
\author{Sandra Cerrai\thanks{Partially supported by the NSF grants DMS-1407615 and DMS-1712934.}\\
\normalsize University of Maryland, College Park\\ United States
\and
Nathan Glatt-Holtz\thanks{Partially supported by the NSF grant DMS-1313272 and the Simons Foundation grant 515990. }\\
\normalsize Tulane University, New Orleans\\ United States}
\date{}

\date{}

\maketitle

\begin{abstract}

We prove the convergence, in the small mass limit, of statistically invariant states for a class of semi-linear damped wave equations, perturbed by an additive Gaussian noise, both with Lipschitz-continuous and with polynomial non-linearities. In particular, we prove that the first marginals of any sequence of invariant measures for the stochastic wave equation converge in a suitable Wasserstein metric  to the unique invariant measure of the limiting stochastic semi-linear parabolic equation obtained in the Smoluchowski-Kramers approximation. The Wasserstein metric is associated to a suitable distance on the space of square integrable functions, that is chosen in such a way that the dynamics of the limiting stochastic parabolic equation is contractive with respect to such a Wasserstein metric. This implies that the limiting result is  a consequence of the validity of a generalized Smoluchowski-Kramers limit at fixed times. The proof of such a generalized limit requires new delicate  bounds for the solutions of the stochastic wave equation, that must be uniform with respect to the size of the mass.

\end{abstract}

\section{Introduction}
\label{sec1}

We are dealing here with the following class of stochastic damped wave equations, defined on a bounded and smooth domain $D\subset \mathbb{R}^d$
\begin{equation}
\label{eq1}
\le\{\begin{array}{l}
\ds{\mu\, \partial^2_t u_\mu(t,\xi)=\Delta u_\mu(t,\xi)-\partial_t u_\mu(t,\xi)+b(\xi,u_\mu(t,\xi))+\partial_tw^Q(t,\xi),}\\
\vs
\ds{u_\mu(0,\xi)=x(\xi),\ \ \partial_t u_\mu(0,\xi)=y(\xi),\ \ \xi \in\,D,\ \ \ \ u_\mu(t,\xi)=0,\ \ \ \ t\geq 0,\ \ \xi \in\,\partial D,}
\end{array}
\r.
\end{equation}
where $0<\mu<<1$. 
Here, $w^Q(t,\xi)$ is a cylindrical Wiener process, white in time and colored in space. The nonlinearity $b:D\times \mathbb{R}\to \mathbb{R}$ is  assumed to be either globally Lipschitz-continuous in the second variable, uniformly with respect to the first one, with its Lipschitz constant dominated by the first eigenvalue of the Laplacian, or only locally Lipschitz continuous and satisfying some polynomial growth and dissipation conditions. In both cases, the identically zero function is globally asymptotically stable in the absence of noise. In the case the nonlinearity $b$ is Lipschitz, we do not impose any restriction on the space dimension $d$, while, in the case of polynomial nonlinearity, we assume that $d=1$.

As a consequence of  Newton's law,  the solution $u_\mu(t,\xi)$ of equation \eqref{eq1} can be interpreted as the displacement field of the particles of a material continuum in a domain $D$, subject to a random external force field $\partial_t w(t,\xi)$ and a damping force proportional to the velocity field (here the proportionality constant is taken equal $1$). The second order differential operator takes into account of the interaction forces between neighboring particles, in  the presence of a non-linear reaction given by $b$. Here  $\mu$ represents the constant density of the particles and we are interested in the regime when $\mu\to 0$, the so-called Smoluchowski-Kramers approximation limit (ref. \cite{kra} and \cite{smolu}).
 
 In \cite{SK1} and \cite{SK2},  the convergence of $u_\mu$ to the solution $u$ of the parabolic problem
 \begin{equation}
\label{eq1-par}
\le\{\begin{array}{l}
\ds{\partial_t u(t,\xi)=\Delta u(t,\xi)+b(\xi,u(t,\xi))+\partial_tw^Q(t,\xi),}\\
\vs
\ds{u(0,\xi)=x(\xi),\ \ \xi \in\,D,\ \ \ \ u(t,\xi)=0,\ \ \ \ t\geq 0,\ \ \xi \in\,\partial D,}
\end{array}
\r.
\end{equation}
has been studied, under analogous conditions on the non-linearity $b$. Namely, it has been  proven that for every $T>0$ and $\e>0$
\begin{equation}
\label{intro2}
\lim_{\mu\to 0}\mathbb{P}\le(\sup_{t \in\,[0,T]}|u_\mu(t)-u(t)|_{L^2(D)}>\e\r)=0.
\end{equation}
In fact, in \cite{cs} it has been shown that, under the condition of Lipschitz-continuity for $b$, the following stronger convergence holds
\begin{equation}
\label{intro1}
\lim_{\mu\to 0}\mathbb{E} \sup_{t \in\,[0,T]}|u_\mu(t)-u(t)|_{L^2(D)}^p=0,
\end{equation}
for every $p\geq 1$.  Note that this type of limit has been addressed in a variety of finite and infinite dimensional context (ref. \cite{bhw}, \cite{SK3}, \cite{f}, \cite{fh}, \cite{hhv}, \cite{hmdvw}, \cite{hu}, \cite{lee}, \cite{spi} for the finite dimension and \cite{SK1},  \cite{SK2},  \cite{CFS}, \cite{sal}, \cite{sal2}, \cite{cs}, \cite{Lv2}, \cite{Lv3}, \cite{Lv4}, \cite{Lv1}, \cite{Nguyen}, \cite{salins18} for the 
infinite dimension).

\medskip

However, once one has proved the validity of the small mass limits \eqref{intro2} and \eqref{intro1} on a fixed time interval $[0,T]$, it is of interest to compare the long-time behavior of the second order system \eqref{eq1}, with that of the first order system \eqref{eq1-par}. In particular, it is desirable to identify conditions for the convergence of statistically steady states. 

In \cite{SK1},   the long time behavior
of equations \eqref{eq1} and \eqref{eq1-par} has been compared, under the assumption that the two systems are of gradient type.
Actually, in the case of white noise in space and time (that is $Q=I$) and dimension $d=1$, an
explicit expression for the Boltzman distribution of the process
$z_\mu(t):=(u_{\mu}(t), \partial u_{\mu}/\partial t(t))$ in the phase space $\H:=L^2(0,L)\times H^{-1}(0,L)$ has been
given. Of course, since in the functional space $\H$ there is no
analogous of the Lebesgue measure,
 an auxiliary Gaussian measure has been introduced, with respect to which
 the density of the Boltzman distribution has  been written. This
auxiliary Gaussian measure is the stationary measure of the linear
wave equation related to problem \eqref{eq1}.
In particular, it has been shown  that  the first marginal of the invariant measure
associated with the process $z_{\mu}(t)$ does not depend on $\mu>0$ and  coincides with the
invariant measure for the heat equation \eqref{eq1-par}.

\medskip

In the case of non-gradient systems, that is when we are not dealing with space-time white noise,  there is no  explicit expression for the invariant measure $\nu_\mu$ associated with system \eqref{eq1} and there is no reason to expect that the first marginal of $\nu_\mu$ does not depend on $\mu$ or  coincides with the invariant measure $\nu$ of system \eqref{eq1-par}. 

Nevertheless, in the present paper we are going to show that in the limit, as the mass $\mu$ goes to zero,  the first marginal of any  invariant measure $\nu_\mu$ of the second order system \eqref{eq1} approaches in a suitable manner  the invariant measure $\nu$ of the first order system \eqref{eq1-par}. Namely, we are going to prove that 
\begin{equation}
\label{intro4}
\lim_{\mu\to 0} \rho_\a\le((\Pi_1\nu_\mu)^\prime,\nu\r)=0,
\end{equation}
where  $(\Pi_1\nu_\mu)^\prime$ is the extension to $L^2(D)$ of the first marginal of the invariant measure $\nu_\mu$ in $H^1(D)\times L^2(D)$ and $\rho_\a$ is the Wasserstein metric on $P(L^2(D))$ associated with a suitable distance $\a$ on $L^2(D)$, to be determined on the basis of the type of non-linearity $b$ we are dealing with.

Actually, the fundamental fact that we need to be true in order to be able to prove   \eqref{intro4} is that there exist a distance $\a$ in $L^2(D)$ and a constant $\d>0$ such that 
\begin{equation}
\label{intro5}
\rho_\a\le(P_t^\star \nu_1,P_t^\star \nu_2\r)\leq c\,e^{-\delta t}\rho_\a(\nu_1,\nu_2),\ \ \ \ t\geq 0,
\end{equation}
for every $\nu_1, \nu_2 \in\,\mathcal{P}r(H)$. Here $P^\star_t$ is the adjoint of the transition semigroup associated with equation \eqref{eq1-par}. As shown in \cite{HM},  a possible proof of \eqref{intro5} is based on a suitable irreducibility condition for system \eqref{eq1-par}, together with suitable large-time smoothing estimates for the Markovian dynamics related closely related to the  so-called {\em asymptotic strong Feller property} (see \cite{HM06}), along with  some estimates for  the solution of the first variation equation of \eqref{eq1-par}. In case of a polynomial non-linearity $b$, the irreducibility condition is uniform with respect to the initial condition, so that we can take $\a$ equal to the usual distance in $L^2(D)$. In case of a Lipschitz $b$, the irreducibility condition is uniform only with respect to initial conditions on  bounded sets of $L^2(D)$, so that a Lyapunov structure has to be taken into consideration. This means in particular that we can prove \eqref{intro5} by using the method developed in \cite{HM}, once we choose
\[\a(x,y)=\inf_{\gamma}\int_0^1 \exp\le(\eta\,|\gamma(t)|^2_H\r)\,|\gamma^\prime(t)|_H\,dt,\]
 where the infimum is taken over all paths $\gamma$ such that $\gamma(0)=x$ and $\gamma(1)=y$ and where $\eta$ is some positive constant to be determined.

Next, as a consequence of \eqref{intro5}, due to the invariance of $\nu_\mu$ and $\nu$, we have
\[\rho_\a\le((\Pi_1\nu_\mu)^\prime,\nu\r)\leq \rho_\a\le((\Pi_1(P_t^\mu)^\star \nu_\mu)^\prime,P_t^\star(\Pi_1\nu_\mu)^\prime\r)+c\,e^{-\delta t}\rho_\a\le((\Pi_1\nu_\mu)^\prime,\nu\r).\]
Therefore, if we pick $t_\star>0$ such that $c\,e^{-\delta t_\star}<1/2$, we have
\[\rho_\a\le((\Pi_1\nu_\mu)^\prime,\nu\r)\leq 2\,\rho_\a\le((\Pi_1(P_t^\mu)^\star \nu_\mu)^\prime,P_t^\star(\Pi_1\nu_\mu)^\prime\r),\ \ \ \ t\geq t_\star.\]
Since the Wasserstein distance $\rho_\a$ between two measures $\nu_1$ and $\nu_2$ is given as
\[\rho_\a(\nu_1,\nu_2)=\inf \mathbb{E}\,\a(X_1,X_2),\]
where $\mathcal{L}(X_i)=\nu_i$, for $i=1,2$, we have
\[\rho_\a\le((\Pi_1(P_t^\mu)^\star \nu_\mu)^\prime,P_t^\star(\Pi_1\nu_\mu)^\prime\r)\leq \mathbb{E}\,
\a(u_\mu^{\gamma_\mu}(t),u^{\theta_\mu}(t)),\]
where $\gamma_\mu=(\theta_\mu,\eta_\mu)$ is a $H^1(D)\times L^2(D)$-valued random variable, distributed as $\nu_\mu$. In particular, this implies that the proof of  \eqref{intro4} reduces to the proof of the following limit
\begin{equation}
\label{intro6}
\lim_{\mu \to 0} \mathbb{E}\,
\a(u_\mu^{\gamma_\mu}(t),u^{\theta_\mu}(t))=0,\end{equation}
 for a fixed time $t$ sufficiently large. 
 
The limit \eqref{intro6}  seems analogous to the Smoluchowski-Kramers approximation \eqref{intro1} proven in \cite{cs}, at least when $\a$ is the usual distance in $L^2(D)$. In fact, limit \eqref{intro6} is considerably stronger, and more insidious to be proven,  in comparison to  \eqref{intro1}, as \eqref{intro1} is only valid for fixed deterministic initial condition, while \eqref{intro6} has to be verified for random initial conditions, depending on $\mu$ through the invariant measure $\nu_\mu$. As a matter of fact, the proof of \eqref{intro6} is the major challenge of the paper and requires a series of delicate and completely new uniform bounds for  the solution of equation \eqref{eq1}, including exponential moments.

To this purpose, we would like to mention that the convergence of stationary solutions, as $\mu$ goes to zero, has been also studied in \cite{Lv1}, in the case there is  $\sqrt{\mu}$ in front of the noise. This case is of course completely different from the case we are considering here and it is also considerably easier from a mathematical point of view.Actually, it does not require the sophisticated bounds that we have proved in the present paper and does not involve invariant measures.

\medskip

To conclude, we would like to remark that our result about the limiting behavior of stationary solutions of system \eqref{eq1}, in the Smoluchowski-Kramers approximation, is consistent to what  proven in \cite{sal} and \cite{sal2} for the limiting behavior of the quasi-potential $V_\mu(u,v)$, that describes the asymptotics of the exit times and the large deviation principle for the invariant measures in equation \eqref{eq1}.
Actually,   in \cite{sal}  the quasi-potential  $V_\mu(u,v)$ associated with \eqref{eq1} has been explicitly computed under the assumption that system \eqref{eq1} is of gradient type and it has been shown that for every fixed $\mu>0$
\begin{equation}
\label{intro3}
V_\mu(u):=\inf_{v \in\,H^{-1}(D)} V_\mu(u,v)=V(u),
\end{equation}
where $V(u)$ is the quasi-potential associated with \eqref{eq1-par}.
Moreover, in \cite{sal2} the non-gradient case has been considered and it has been proven that in spite of the fact that there is no explicit expression for $V_\mu(u,v)$, nevertheless
\[\lim_{\mu\to 0}V_\mu(u)=V(u),\]
for every sufficiently regular $u$.

\bigskip

{\em Organization of the paper:} In Section \ref{sec2}, we introduce all of our notations and assumptions on the non-linearity $b$ and on the noise $w^Q$ and we discuss some  consequences of these choices. In Section \ref{sec4} we deal with the semigroup $S_\mu(t)$ associated with the linear damped wave equation and we give  some generalization of results proved in previous papers. Such generalizations are needed to prove the a-priori bounds required in the proof of \eqref{intro6}.  In Section \ref{sec3} we collect some results about equations \eqref{eq1} and \eqref{eq1-par} and their transition semigroups. In Section \ref{sec5} we give the main result of the paper and we introduce the method of the proof. In Section \ref{sec6} we prove limit \eqref{intro6} and in the following Sections \ref{sec7}, \ref{sec8} and \ref{sec9} we give a proof of all the  lemmas needed in the proof of limit \eqref{intro6}. Finally, in Appendix \ref{appa} we recall for the reader's convenience the proof of the asymptotic strong Feller property for equation \eqref{eq1-par}.

 \section{Assumptions and  notations }

 \label{sec2}
 
 We assume that $D$ is a bounded open domain in $\mathbb{R}^d$, having a smooth boundary. We denote by $H$ the Hilbert space $L^2(D)$, endowed with the usual scalar product $\langle\cdot,\cdot\rangle_H$ and the corresponding norm $|\cdot|_H$.
The norm in $L^p(D)$ is denoted by $|\cdot|_{L^p}$.

 We denote by $A$ the realization of the Laplace operator in $H$, endowed with  Dirichlet boundary conditions. It is known that there exists a complete orthonormal system $\{e_k\}_{k \in\,\mathbb{N}}\subset H$ and a sequence of positive real numbers $\{\a_k\}_{k \in\,\mathbb{N}}$, diverging to $\infty$  such that
 \begin{equation}
 \label{f2002}
 A e_k =-\a_k e_k,\ \ \ \ \ k \in\,\mathbb{N}.\end{equation}

 Next, for any $\beta \in\,\mathbb{R}$, we denote by $H^\beta$ the completion of $C_0(D)$ with respect to the norm
 \[|x|_{H^\beta}^2=\sum_{k=1}^\infty \a_k^\beta\,\le|\langle x,e_k\rangle_H\r|^2.\]
 Moreover, we denote $\mathcal{H}_\beta=:H^\beta\times H^{\beta-1}$, and for every $(x,y) \in\,\mathcal{H}_\beta$, we set
 \[\Pi_1(x,y)=x,\ \ \ \ \ \Pi_2(x,y)=y.\]
 
 Throughout the paper, for every $\mu>0$ and $\beta \in\,\mathbb{R}$, we define
\begin{equation}
\label{amu}
A_\mu(x,y)=\frac 1\mu\,(\mu y,Ax-y),\ \ \ \ \ (x,y) \in\,D(A_\mu)=\mathcal{H}_{\beta+1}.\end{equation}
It can be proven that $A_\mu$ is the generator of a strongly continuous group of bounded linear operators $\{S_\mu(t)\}_{t\geq 0}$ on each $\mathcal{H}_\beta$ (for a proof see \cite[Section 7.4]{pazy}).

\medskip

In this paper, we assume that the cylindrical Wiener process $w^Q(t,\xi)$ is white in time and colored in space, with spatial covariance $Q^2$, for some  $Q \in\,\mathcal{L}^+(H)$, where $\mathcal{L}^+(H)$ is the space of non-negative and symmetric bounded linear operators on $H$. This means that $w^Q(t,\xi)$ can be formally represented as the sum
\[w^Q(t,\xi)=\sum_{k=1}^\infty Q e_k (\xi) \beta_k(t),\]
where $\{e_k\}_{k \in\,\mathbb{N}}$ is the complete orthonormal system of $H$ that diagonalizes $A$ and $\{\beta_k(t)\}_{k \in\,\mathbb{N}}$ is a sequence of mutually independent Brownian motions, all defined on the same stochastic basis $(\Omega, \mathcal{F}, \{\mathcal{F}_t\}_{t\geq 0}, \mathbb{P})$. 

We impose the following two conditions on $Q$ in what follows. The first condition ensures sufficient spatial smoothness to define solutions in the right functional space and prove a-priori bounds for them. The second condition guarantees that the noise acts directly through a large enough subspace of the phase space, in order to ensure   smoothing in the Markovian dynamics associated to \eqref{eq1-par} (see also remark \ref{RA1})

\begin{Hypothesis}
\label{H1}
\begin{enumerate}
\item There exists a non-negative sequence $\{\la_k\}_{k \in\,\mathbb{N}}$ such that
\begin{equation}
\label{fine333}
Q e_k=\la_k\,e_k,\ \ \ \ k \in\,\mathbb{N}.\end{equation}
Moreover, there exists $\d>0$ such that 
\begin{equation}
\label{ghc301}
\text{{\em Tr}}\, [Q^2(-A)^\d]=\sum_{k=1}^\infty \la_k^2 \a_k^\d <\infty.\end{equation}
\item There exist $\bar{n} \in\,\mathbb{N}$ and $c>0$ such that
\begin{equation}
\label{ghc411}
|Q x|_H\geq c\,|P_{\bar{n}}\,x|_H,\ \ \ \ x \in\,H,
\end{equation}
where $P_n$ is the projection of $H$ onto $H_n:=\text{span}\le\{e_1,\ldots,e_n\r\}$, for every $n \in\,\mathbb{N}$ (concerning $\bar{n}$, see \eqref{sk3}).
\end{enumerate}
\end{Hypothesis}

In what follows, for any $\mu>0$ we define the bounded linear operator
\begin{equation}
\label{qmu}
Q_\mu:H\to \H_1,\ \ \ \ x \in\,H\mapsto \frac1\mu(0,Qx) \in\,\H_1.\end{equation}

Concerning the non-linearity $b$, we assume the following conditions.

\begin{Hypothesis}
\label{H3}
\begin{enumerate}
\item The mapping  $b:D\times \mathbb{R}\to \mathbb{R}$ is measurable and $b(\xi,\cdot):\mathbb{R}\to\mathbb{R}$ is of class $C^1$, for almost all $\xi \in\,D$. Moreover,
\begin{equation}
\label{sk-bound}
\sup_{(\xi,\si) \in\,D\times \mathbb{R}}\partial_\si b(\xi,\si)=:L_b<\infty.
\end{equation}
\item There exists $\la \in\,[1,3]$ such that for every $\si \in\,\mathbb{R}$
\begin{equation}
\label{sk1}
\sup_{\xi \in\,D}|b(\xi,\si)|\leq c\,\le(1+|\si|^\la\r),\ \ \ \ \ \sup_{\xi \in\,D}|\partial_\si b(\xi,\si)|\leq c\,\le(1+|\si|^{\la-1}\r).
\end{equation}
\item If $\la=1$,  we have
\begin{equation}
\label{sk3-tris}
L_b<\a_1.
\end{equation}
\item If $\la \in\,(1,3]$, we have
\begin{equation}
\label{sk3}
L_b<\a_{\bar{n}},
\end{equation}
where $\bar{n}$ is the integer introduced in \eqref{ghc411}.
Moreover, there exist
 $c_1>0$ and $c_2 \in\,\mathbb{R}$  such that for every $(\xi,\si)\in\,D\times \mathbb{R}$
\begin{equation}
\label{sk3-bis}
\sup_{\xi\in\,D} b(\xi,\si)\si \leq -c_1\,|\si|^{\la+1}+c_2.
\end{equation}

\item If $\la \in\,(1,3]$, then $d=1$ and $D=[0,L]$, for some $L>0$.

\end{enumerate}

\end{Hypothesis}

\begin{Remark}
{\em 
\begin{enumerate}
\item  Let $a:D\to \mathbb{R}$ be a measurable function such that 
\[0<a_0\leq a(\xi)\leq a_1,\ \ \ \ \ \xi \in\,D,\]
and let $f:D\times \mathbb{R}\to\mathbb{R}$ be a measurable mapping such that $f(\xi,\cdot):\mathbb{R}\to\mathbb{R}$ is Lipschitz continuous, uniformly with respect to $\xi \in\,{D}$.
Then, if $\la \in\,(1,3]$, the function $b$ defined by
\[b(\xi,\si)=-a(\xi)|\si|^{\la-1}\si+f(\xi,\si),\ \ \ \ (\xi,\si) \in\,D\times \mathbb{R},\]
satisfies all conditions in Hypothesis \ref{H3}. In case $f=0$, $b$ is the so-called Klein-Gordon nonlinearity. 
\item If $\la=1$, for every $\e>0$ we can find $c_\e>0$ such that
\begin{equation}
\label{ghc421}
\sup_{\xi \in\,D}b(\xi,\si)\si\leq (L_b+\e)\,|\si|^2+c_\e,\ \ \ \ \ \si \in\,\mathbb{R},
\end{equation}
where $L_b$ is the constant introduced in \eqref{sk-bound}.
\item If $\la \in\,(1,3]$, from \eqref{sk1} and \eqref{sk3-bis} it  follows that there exist $\kappa_1>0$ and $\kappa_2 \in\,\mathbb{R}$  such that for every $\si, \rho \in\,\mathbb{R}$
\begin{equation}
\label{ghc401}
\sup_{\xi \in\,D} b(\xi,\si+\rho)\si\leq -\kappa_1\,|\si|^{\la+1}+\kappa_2\le(1+|\rho|^{\la+1}\r).
\end{equation}
\item If $\la=1$, as a consequence of \eqref{sk3-tris}, it is possible to prove that for every $\xi \in\,D$ there exists an antiderivative $\mathfrak{b}(\xi,\cdot)$ of $b(\xi,\cdot)$ such that
\begin{equation}
\label{ghc420}
\sup_{\xi \in\,D}\mathfrak{b}(\xi,\si)\leq L_b\,|\si|^2,\ \ \ \ \ \si \in\,\mathbb{R},
\end{equation}
where $L_b$ is the constant introduced in \eqref{sk-bound}.
\item In the same way, when $\la \in\,(1,3]$, due to \eqref{sk3} and \eqref{sk3-bis} it is possible to prove that for every $\xi \in\,D$ there exists an antiderivative $\mathfrak{b}(\xi,\cdot)$ of $b(\xi,\cdot)$ such that
\begin{equation}
\label{sk2}
\sup_{\xi \in\,D} \mathfrak{b}(\xi,\si) \leq -\kappa\,|\si|^{\la+1},\ \ \ \ \si \in\,\mathbb{R},
\end{equation}
for some constant $\kappa>0$.

\end{enumerate}

\qed}
\end{Remark}

\bigskip

Now, for every $x \in\,H$, we define
\[B(x)(\xi)=:b(\xi,x(\xi)),\ \ \ \ \xi \in\,D.\]
Notice that if $\la=1$, then $b(\xi,\cdot):\mathbb{R}\to\mathbb{R}$ is Lipschitz-continuous, uniformly with respect to $\xi \in\,D$, so that $B:H\to H$ is Lipschitz-continuous. 
Moreover, as shown in \cite[Theorem 4]{RS}, for every $u \in\,{H^\beta}$, with $0\leq \beta\leq 1$, we have
\begin{equation}
\label{ghc513}
|B(u)|_{H^\beta}\leq c\le(1+|u|_{H^\beta}\r).\end{equation}
Notice moreover that in the general case, due to \eqref{sk1}, for every $u \in\,H^1\cap L^\infty(D)$
\begin{equation}
\label{fine339}
|B(u)|_{H^1}\leq |\partial_\si b(\cdot,u)|_{L^\infty}|u|_{H^1}\leq c\le(1+|u|_{L^\infty}^{\la-1}\r)|u|_{H^1}.
\end{equation}

Next,
for every $\mu>0$ and $\beta \in\,[0,1]$, we set
\begin{equation}
\label{bmu}
B_\mu(x,y)(\xi)=:\frac 1\mu (0,B(x))(\xi)=\frac 1\mu \,\le(0,b(\xi,x(\xi)\r),\ \ \ \ \xi \in\,D,\ \ \ \ (x,y) \in\,\mathcal{H}_\beta.\end{equation}

If $\la=1$, then  the Lipschitz continuity of $B:H\to H$ implies that, for every $z_1=(x_1,y_1)$ and $z_2=(x_2,y_2)$ in  $\mathcal{H}_\beta$, 
\[\begin{array}{l}
\ds{|B_\mu(z_1)-B_\mu(z_2)|_{\mathcal{H}_\beta}=\frac 1\mu\,|B(x_1)-B(x_2)|_{H^{\beta-1}}\leq \frac c\mu\,|B(x_1)-B(x_2)|_{H}}\\
\vs
\ds{\leq \frac{cM}\mu\,|x_1-x_2|_H\leq \frac{cM}\mu\,|z_1-z_2|_{\mathcal{H}_\beta}.}
\end{array}\]
This means that the mapping $B_\mu:\mathcal{H}_\beta\to \mathcal{H}_\beta$ is Lipschitz continuous, for every $\beta \in\,[0,1]$ and $\mu>0$.

If $\la \in\,(1,3]$, the mapping $b(\xi,\cdot):\mathbb{R}\to \mathbb{R}$ is only locally Lipschitz-continuous and does not have sub-linear growth in general. 
Since
\[\sup_{(\xi,\si) \in\,D\times \mathbb{R}}\,\frac{|\mathfrak{b}(\xi,\si)|}{(|\si|^{\la+1}+1)}<\infty,\]
due to \eqref{sk2}, we have
\begin{equation}
\label{ghc140}
-c\,\le( |\si|^{\la+1}+1\r)\leq \mathfrak{b}(\xi,\si)\leq -\kappa\,|\si|^{\la+1},\ \ \ \ (\xi,\si) \in\,D\times \mathbb{R}.
\end{equation}

On the other hand, if $\la=1$, due to \eqref{ghc420} we have
\begin{equation}
\label{ghc430}
-c\,\le(1+|\si|^2\r)\leq \mathfrak{b}(\xi,\si)\leq L_b\,|\si|^2,\ \ \ \ \ (\xi,\si) \in\,D\times \mathbb{R},
\end{equation}
where $L_b$ is the constant introduced in \eqref{sk3-tris}.

Recall that when $\la \in\,(1,3]$ we assume $D=[0,L]$, so that  $H^1\hookrightarrow L^\infty(D)\hookrightarrow L^p(D)$, for every $p\geq 1$. This implies, together with \eqref{sk1}, that for every $z=(x,y) \in\,\mathcal{H}_1$
\[|B_\mu(z)|_{\mathcal{H}_1}=\frac 1\mu\,|b(\cdot,x)|_H\leq \frac c\mu\,\le(1+|x|^\la_{L^{2\la}(D)}\r)\leq \frac c\mu \le(1+|x|^\la_{H^1}\r)\leq \frac c\mu \le(1+|z|^\la_{\mathcal{H}_1}\r),\]
for every $\mu>0$, for some $c$ independent of $\mu$.
Moreover, by the same arguments, for every $z_1, z_2 \in\,\mathcal{H}_1$
\begin{equation}
\label{fine334}
|B_\mu(z_1)-B_\mu(z_2)|_{\mathcal{H}_1}\leq \frac c\mu\,|x_1-x_2|_{H^1}\le(1+|x_1|_{H^1}^{\la-1}+|x_2|_{H^1}^{\la-1}\r).\end{equation}
In particular, $B_\mu:\mathcal{H}_1\to\mathcal{H}_1$ is locally Lipschitz-continuous.

\section{The approximating semigroup $S_\mu(t)$}
\label{sec4}

In the present section, we present some modifications of the results proven in \cite{cs} about the approximating semigroup $S_\mu(t)$, in the slightly different setting where the motion of the particle with small mass is subject to a magnetic field and  a friction. These modifications are necessary here because, in order to prove \eqref{ghc416} we need some uniformity with respect to the initial conditions on bounded sets.

For every $\mu>0$, $\beta \in\,\mathbb{R}$ and $(x,y) \in\,\mathcal{H}_\beta$, we shall denote
\[u_\mu(t):=\Pi_1 S_\mu(t)(x,y),\ \ \ \ \ \ v_\mu(t):=\Pi_2 S_\mu(t)(x,y),\ \ \ \ \ t\geq 0,\]
where $S_\mu(t)$ is the semigroup generated in $\mathcal{H}_\beta$ by the operator $A_\mu$ and introduced in \eqref{amu}.

In \cite[Lemma 3.1]{cs}, it has been shown that for every $\mu>0$, $\beta \in\,\mathbb{R}$ and $(x,y) \in\,\mathcal{H}_\beta$ 
\begin{equation} \label{energy-est-1}
  \mu \left|v_\mu(t) \right|_{H^{\beta-1}}^2 + |u_\mu(t)|_{H^\beta}^2+2\int_0^t|v_\mu(s)|_{H^{\beta-1}}^2\,ds
 = \mu|y|_{H^{\beta-1}}^2 + |x|^2_{H^\beta},
\end{equation}
and
\begin{equation} \label{energy-est-2}
  \mu |u_\mu(t)|_{H^\beta}^2 + \left| \mu v_\mu(t) + u_\mu(t) \right|_{H^{\beta-1}}^2+2\int_0^t|u_\mu(s)|_{H^\beta}^2\,ds
=\mu |x|_{H^\beta}^2 + |\mu y + x|_{H^{\beta -1}}^2.
\end{equation}
In particular, the two equalities above imply that for any $\mu>0$ there exists $c_{\mu}>0$ such that for any $(x,y) \in\,\mathcal{H}_\beta$
\[\int_0^\infty |S_\mu(t)(x,y)|_{\mathcal{H}_\beta}^2\,dt\leq \frac {c_{\mu}}{2}|(x,y)|_{\mathcal{H}_\beta}^2,\]
and, as a consequence of the Datko theorem, we can conclude that there exist $M_{\mu,\beta},$ and  $\omega_{\mu,\beta}>0$ such that
\[\|S_\mu(t)\|_{{\cal L}(\mathcal{H}_\beta)}\leq M_{\mu,\beta}\,e^{-\omega_{\mu,\beta}t},\ \ \ \ t\geq 0.\]

Moreover, in \cite[Lemma 3.2]{cs} it has been proven that 
for any $\mu>0$,  and for any $\beta \in \reals$ it holds
\begin{equation}
\label{s2}
\left|\Pi_1 S_\mu(t)(0,y) \right|_{H^\beta} \leq 2\, \mu\, |y|_{H^{\beta}},\ \ \ t\geq 0,\ \ \ \ y \in\,H^{\beta}.\end{equation}

The following result has been proven in \cite[Lemma 3.3]{cs}. Notice that in \cite{cs}, estimate \eqref{bound-1} is given only for $\beta=0$, but the case $\beta>0$ is an easy generalization and we omit its proof.

\begin{Lemma}
\label{l3}
For any $0<\rho<1$ there exists a constant $c_\rho>0$ such that for any $k \in \nat$ and $\beta\geq 0$
  \begin{equation} \label{bound-1}
    \sup_{\mu>0} \int_0^\infty s^{-\rho} \left| \Pi_1 S_\mu(s) Q_\mu e_k \right|_{H^\beta}^2 ds \leq c_\rho\, \frac{\la_k^2}{\alpha_k^{1-(\rho+\beta)}},
  \end{equation}
\end{Lemma}
where $Q_\mu$  is as in \eqref{qmu} and $\la_k$ and $\a_k$ denote the eigenvalues of $Q$ and $A$, respectively. 
In particular, \eqref{bound-1} implies that for every $p\geq 1$, $T>0$ and $\d>0$
\begin{equation}
\label{ghc303}
\sum_{k=1}^\infty \frac{\la_k^2}{\a_k^{1-\d}}<\infty \Longrightarrow \sup_{\mu>0}\,\mathbb{E}\sup_{t \in\,[0,T]}|\Pi_1 \Gamma_\mu(t)|^p_{H^\beta}<+\infty,\ \ \ \ \beta<\d,
\end{equation}
where $\Gamma_\mu(t)$ is the stochastic convolution as in \eqref{ghc3}.

\medskip

Next, we denote by $S(t)$, $t\geq 0$, the semigroup generated by the operator $A$ in $H^\beta$, for every $\beta \in\,\mathbb{R}$.
The following two approximation results are a  modification of what proven in \cite[Theorem 3.5 and Corollaries 3.6 and 3.7]{cs}. Such a modification is critical for the proof of Theorem \ref{fund}, our main finite-time convergence result below.

\begin{Lemma}
\label{l4}
Let $T>0$ and $\beta>0$ be fixed. Then, for every $R>0$ it holds
  \begin{equation} 
  \label{Smu-to-Se-at-a-point}
\lim_{\mu\to 0}\ \sup_{|x|_{H^\beta}+\sqrt{\mu}\,|y|_{H}\leq R}\ 
 \sup_{t \leq T} |\Pi_1 S_\mu(t)(x,y) - S(t) x|_{H}=0.
  \end{equation}
  \end{Lemma}
  
 \begin{proof}
If we denote by $P_n$ the projection of $H$ onto span$\{e_1,\ldots,e_n\}$, we have
\[\begin{array}{l}
\ds{ |\Pi_1 S_\mu(t)(x,y) - S(t) x|_{H}\leq  |\Pi_1 S_\mu(t)(0,y)|_H + |\Pi_1 S_\mu(t)(P_n x,0) - S(t) P_n x|_{H}}\\
\vs
\ds{+ |\Pi_1 S_\mu(t)(x-P_n  x,0) |_{H}+ |S(t) (x-P_nx)|_{H}=:I^\mu_1(t)+ I^\mu_{2,n}(t)+I_{3,n}^\mu(t)+I_{4,n}(t).}
\end{array}\] 
As a  consequence of \eqref{s2}, for every $t\geq 0$ and $\mu>0$ we have
\begin{equation}
\label{ghc501}
I_1^\mu(t)\leq 2\,\mu\,|y|_{H}.
\end{equation}
Moreover, as shown in \cite[proof of Theorem 3.5]{cs}, for every $n \in\,\mathbb{N}$ there exists some $c_T(n)$ such that for every $t \in\,[0,T]$ and $\mu>0$
\begin{equation}
\label{ghc502}
I^\mu_{2,n}(t)\leq \mu\,c_T(n)\,|x|_H.
\end{equation}
Moreover,  as a consequence of \eqref{energy-est-1}, for every $n \in\,\mathbb{N}$, $t \in\,[0,T]$ and $\mu>0$  we have
\[I_{3,n}^\mu(t)+I_{4,n}(t)\leq c_T\,|P_n x-x|_H\leq c_T\,\|P_n-I\|_{\mathcal{L}(H^\beta,H)}|x|_{H^\beta}.
\]
In particular, if we fix $R>0$ and $\e>0$, we can find $\bar{n}=n(\e,R,T)$ such that
\[|x|_{H^\beta}\leq R\Longrightarrow I_{3,\bar{n}}^\mu(t)+I_{4,\bar{n}}(t)\leq \e,\ \ \ \ \mu >0.\]
Thanks to \eqref{ghc501} and \eqref{ghc502}, this implies that
\[|x|_{H^\beta}+ \sqrt{\mu}\,|y|_{H}\leq R \Longrightarrow \sup_{t \in\,[0,T]}\, |\Pi_1 S_\mu(t)(x,y) - S(t) x|_{H}\leq 2\,\sqrt{\mu}\,R+\mu\,c_T(\bar{n})R+\e,\]
and, due to the arbitrariness of $\e$,  \eqref{Smu-to-Se-at-a-point} follows.

 \end{proof}

\begin{Lemma}

Let $T>0$ and $\beta>0$ be fixed. Then, for any $R>0$
  \begin{equation} \label{Smu-to-Se-2}
\lim_{\mu\to 0}\,\sup_{|\psi|_{L^1(0,T;H^\beta)}\leq R}\ \sup_{t \in\,[0,T]}\ \left|\frac{1}{\mu} \int_0^t \Pi_1 S_\mu(t-s) (0,\psi(s)) ds - \int_0^t S(t-s)  \psi(s) ds \right|_{H} =0.
  \end{equation}
\end{Lemma}

\begin{proof}
For every $\mu>0$ and $t\geq 0$, we define
\begin{equation}
\label{ghc534}
\Phi_\mu(t)y=\frac 1\mu \Pi_1S_\mu(t)(0,y)-S(t)y,\ \ \ \ y \in\,H.
\end{equation}
Then,   for every $n \in\,\mathbb{N}$ we have
\[\begin{array}{l}
\ds{\frac{1}{\mu} \int_0^t \Pi_1 S_\mu(t-s) (0,\psi(s)) ds - \int_0^t S(t-s)  \psi(s) ds }\\
\vs
\ds{=\int_0^t \Phi_\mu(t-s)\le[\psi(s)-\psi_n(s)\r]\,ds+\int_0^t \Phi_\mu(t-s)\psi_n(s)\,ds,}
\end{array}\]
where
\[\psi_n(s):=I_{\{|\psi(s)|_H\leq n\}}\,P_n\psi(s), \ \ \ \ s\geq 0.\]
As shown in \cite[proof of Theorem 3.5]{cs}, for every $n \in\,\mathbb{N}$, there exists $c_T(n)$ such that for every $t \in\,[0,T]$
\begin{equation}
\label{ghc505}
|\Phi_\mu(t)P_ny|_H\leq c_T(n)\le(e^{-\frac t\mu}+\mu\r)|y|_H.
\end{equation}
Moreover, due to \eqref{s2}, for every $\beta \in\,[0,1]$, we have
\begin{equation}
\label{ghc512}
\|\Phi_\mu(t)\|_{\mathcal{L}(H^\d)}\leq M_{\beta,T},\ \ \ \ t \in\,[0,T].
\end{equation}
Hence 
\[\begin{array}{l}
\ds{\left|\frac{1}{\mu} \int_0^t \Pi_1 S_\mu(t-s) (0,\psi(s)) ds - \int_0^t S(t-s)  \psi(s) ds \right|_{H} }\\
\vs
\ds{\leq c\,M_{0,T}\int_0^t|\psi(s)-\psi_n(s)|_H\,ds+c_{T}(n)\int_0^t\le(e^{-\frac{t-s}{\mu}}+\mu\r)I_{\{|\psi(s)|_H\leq n\}}|\psi(s)|_H\,ds}\\
\vs
\ds{\leq c\,M_{0,T}\,\|P_n-I\|_{\mathcal{L}(H^\beta,H)}\,|\psi|_{L^1(0,T;H^\beta)}+n\,c_{T}(n)\le(1+T\r) \mu.}
\end{array}\]
Thus, if we  fix $\e>0$ and pick $\bar{n}=n(T,R,\e) \in\,\mathbb{N}$ such that
\[c\,M_{0,T}\,\|P_{\bar{n}}-I\|_{\mathcal{L}(H^\beta,H)}\,R\leq \e,\]
we get
\[\begin{array}{l}
\ds{\limsup_{\mu\to 0}\,\sup_{|\psi|_{L^1(0,T;H^\beta)}\leq R}\,\left|\frac{1}{\mu} \int_0^t \Pi_1 S_\mu(t-s) (0,\psi(s)) ds - \int_0^t S(t-s)  \psi(s) ds \right|_{H}}\\
\vs
\ds{\leq \e+\lim_{\mu\to 0}\bar{n}\,c_{T}(\bar{n})\le(1+T\r) \mu=\e.}
\end{array}\]
As $\e$ is arbitrary,  this yields \eqref{Smu-to-Se-2}.

\end{proof}

\section{Preliminary results on equations \eqref{eq1} and \eqref{eq1-par}}
\label{sec3}

If we denote $z(t)=(u(t),\partial_t u(t))$ and 
\[w(t,\xi)=\sum_{k=1}^\infty e_k(\xi)\,\beta_k(t),\]
with the notation that we have just introduced in Section \ref{sec2}, equation \eqref{eq1} can be rewritten as the following abstract stochastic evolution equation in $\mathcal{H}_\beta$
\begin{equation}
\label{eq1-abstr}
dz_\mu(t)=\le[A_\mu z_\mu(t)+B_\mu(z_\mu(t))\r]\,dt+Q_\mu dw(t),\ \ \ \ z_\mu(0)=(x,y) \in\,\H_\beta.
\end{equation}
In case $\la=1$, if
\[\sum_{k=1}^\infty \frac{\la_k^2}{\a_k^{1-\d}}<\infty,\]
for some $\d \in\,[0,1]$, then for every  $\beta<\d$ and any random initial condition $\gamma=(\theta,\eta) \in\,L^p(\Omega;\mathcal{H}_\beta)$ equation \eqref{eq1-abstr} admits a unique mild solution $z_\mu^\gamma=(u_\mu^\gamma,\partial_t u_\mu^\gamma) \in\,L^p(\Omega;C([0,T];\mathcal{H}_\beta))$, for every $p\geq 1$, $T>0$ and $\mu>0$. This means that
\begin{equation}
\label{fine335}
z_\mu^\gamma(t)=S_\mu(t) \gamma+\int_0^tS_\mu(t-s) B_\mu(z_\mu^\gamma(s))\,ds+\Gamma_\mu(t),\end{equation}
where
\begin{equation}
\label{ghc3}
\Gamma_\mu(t):=\int_0^tS_\mu(t-s) Q_\mu dw(s),
\end{equation}
(for a proof see e.g. \cite[Theorem 7.2]{dpz1}).

In \cite{bdp} and \cite{chow}   it is proven that in case $\la \in\,(1,3]$, for every $\mu>0$ and $\gamma=(\theta,\eta) \in\,L^p(\Omega;\mathcal{H}_1)$ there exists a unique mild solution $z_\mu^\gamma=(u_\mu^\gamma,\partial_t u_\mu^\gamma) \in\,L^p(\Omega;C([0,T];\mathcal{H}_1)$, for every $p\geq 1$ and $T>0$.

In what follows, we shall denote by $P^\mu_t$ the transition semigroup associated with equation \eqref{eq1-abstr} in the space $\mathcal{H}_1$. Namely, 
\[P^\mu_t\varphi(z)=\mathbb{E}\,\varphi(z^z_\mu(t)),\ \ \ \ t\geq 0,\ \ \ z\in\,\mathcal{H}_1,\] 
for every Borel bounded $\varphi:\mathcal{H}_1\to \mathbb{R}$.
In \cite{bdp} and \cite{chow}, the long time behavior of equation \eqref{eq1-abstr} has been studied and it has been proven that,  under Hypotheses \ref{H1} and  \ref{H3}, for every fixed $\mu>0$ the semigroup $P^\mu_t$ admits an invariant measure $\nu_\mu$.
In both papers conditions for uniqueness of the invariant measure are also given.

\medskip

Next, we consider equation \eqref{eq1-par}.
If Hypotheses \ref{H1} and  \ref{H3} hold, then for any $T>0$, $p\geq 1$ and $\theta \in\,L^p(\Omega;H)$,  equation \eqref{eq1-par} admits a unique mild solution $u ^\theta \in\,L^p(\Omega;C([0,T];H))$. This means that if $S(t)$ denotes the semigroup generated by the operator $A$ in $H$, then
\[u^\theta (t)=S(t)\theta +\int_0^t S(t-s)B(u^\theta(s))\,ds+\Gamma(t),\]
where
\begin{equation}
\label{ghc4}
\Gamma(t):=\int_0^t S(t-s)dw^Q(s).
\end{equation}
Moreover, it can be shown that
\begin{equation}
\label{ghc533}
\mathbb{E}\,\sup_{t \in\,[0,T]}|u^\theta(t)|^2_{H}\leq c_{T}\le(1+\mathbb{E}\,|\theta|^2_{H}\r).\end{equation}
(for a proof, see e.g. \cite{dpz1}).

If $\la=1$, then it is possible to show that in fact for every $\beta \in\,[0,1]$ and  $\theta \in\,L^p(\Omega;H^\beta)$ equation \eqref{eq1-par} is well posed in $L^p(\Omega;C([0,T];H^\beta))$ and
\begin{equation}
\label{ghc520}
\mathbb{E}\,\sup_{t \in\,[0,T]}|u^\theta(t)|^p_{H^\beta}\leq c_{\beta,T}\le(1+\mathbb{E}\,|\theta|^p_{H^\beta}\r).
\end{equation}
In case $\la \in\,(1,3]$,  if $\theta \in\,L^p(\Omega;H^1)$, then
\begin{equation}
\label{ghc521}
\mathbb{E}\,\sup_{t \in\,[0,T]}|u^\theta(t)|^p_{H^1}\leq c_{T}\le(1+\mathbb{E}\,|\theta|^p_{H^1}\r)
\end{equation}
(for a proof, see e.g. \cite{dpz1}).

Now, we denote by $P_t$ the transition semigroup associated with equation \eqref{eq1-par} in $H$,
\[P_t\varphi(x)=\mathbb{E} \varphi(u^x(t)),\ \ \ \ \ t\geq 0,\ \ \ \ x \in\,H,\]
for every Borel bounded $\varphi:H\to \mathbb{R}$. It is possible to prove that under Hypotheses \ref{H1} and  \ref{H3}, the semigroup $P_t$ admits an invariant measure $\nu$ (for a proof in the case $\la  \in\,(1,3]$, see \cite[Proposition 8.2.2]{libro}).

\begin{Lemma}
Assume Hypotheses \ref{H1} and  \ref{H3}. Then, for every $M, \e>0$ there exists $t_\star=t_\star(M, \e)>0$ such that
\begin{equation}
\label{ghc400}
\inf_{|x|_H\leq M }\,\mathbb{P}\le(|u^x(t)|_H<\e\r)>0,\ \ \ \ t\geq t_\star.
\end{equation}
Moreover, in the case $\la \in(1,3]$ we have
\begin{equation}
\label{ghc400bis}
\inf_{x \in\,H }\,\mathbb{P}\le(|u^x(t)|_H<\e\r)>0,\ \ \ \ t\geq t_\star,
\end{equation}
for some $t_\star=t_\star(\e)>0$.

\end{Lemma}

\begin{proof}
We give here a proof for \eqref{ghc400bis}, that holds when $\la \in\,(1,3]$. The proof of \eqref{ghc400}  is left to the reader.
We define $v^x(t)=u^x(t)-\Gamma(t)$, where $\Gamma(t)$ is the stochastic convolution defined in \eqref{ghc4}. We have that $v^x$ solves the random equation
\[\frac{d}{dt} v^x(t)=A v(t)+B(v^x(t)+\Gamma(t)),\ \ \ \ \ v^x(0)=x.\]
Due to \eqref{ghc401}, we have
\[\begin{array}{l}
\ds{\frac 12 \frac d{dt}|v^x(t)|_H^2=\langle Av^x(t),v^x(t)\rangle_H+\langle B(v^x(t)+\Gamma(t)),v^x(t)\rangle_H}\\
\vs
\ds{\leq -\kappa_1 |v^x(t)|_{L^{\la+1}}^{\la+1}+\kappa_2\le(1+|\Gamma(t)|_{L^{\la+1}}^{\la+1}\r)\leq -\kappa_1\,|D|^{\frac{\la-1}2}\, |v^x(t)|_{H}^{\la+1}+\kappa_2\le(1+|\Gamma(t)|_{L^{\la+1}}^{\la+1}\r).}
\end{array}\]
Since $\la+1>2$, by a comparison argument (see \cite[Lemma 1.2.6]{libro}) the inequality above yields
\[\sup_{x \in\,H}|v^x(t)|_H\leq c\,\max\le(\sup_{s \in\,[0,t]}|\Gamma(s)|^{\frac 1{\la+1}}_{L^{\la+1}},\ t^{-\frac {\la-1}4}\r),\ \ \ \ t>0,\]
for some constant $c>0$ depending  on $\la$.
This implies that there exists $t_\star=t_\star(\e)>0$ such that
\[\begin{array}{l}
\ds{\inf_{x \in\,H}\mathbb{P}\le(|v^x(t)|_H<\e/2\r)
\geq \mathbb{P}
\le(\sup_{s \in\,[0,t]}|\Gamma(s)|_{L^{\la+1}}<\le(\frac \e{2c}\r)^{\la+1}\r),\ \ \ \ t\geq t_\star.}
\end{array}\]
Since $|u^x(t)|_H\leq |v^x(t)|_H+|\Gamma(t)|_H\leq |v^x(t)|_H+|D|^{\frac{\la-1}{2(\la+1)}}\,|\Gamma(t)|_{L^{\la+1}}$ and we can assume $\e\leq 1$, from the inequality above we obtain
\[\begin{array}{l}
\ds{\inf_{x \in\,H}\mathbb{P}\le(|u^x(t)|_H<\e\r)
\geq \mathbb{P}
\le(\sup_{s \in\,[0,t]}|\Gamma(s)|_{L^{\la+1}}<c\,\e^{\la+1}\r),\ \ \ \ t\geq t_\star,}
\end{array}\]
for some constant $c>0$. Since the law of $\Gamma(\cdot)$ is full in $C([0,t];L^{\la+1}(D))$, this implies \eqref{ghc400bis}. 

\end{proof}

\medskip

In \cite{libro} the {\em strong Feller property} of the semigroup $P_t$ is studied  and it proven that, under suitable non-degeneracy conditions on the covariance $Q^2$, 
for every $\varphi \in\,B_b(H)$ and $t>0$ we have $P_t\varphi \in\,C^1_b(H)$, and
\[\sup_{x \in\,H}|D(P_t\varphi)(x)|_H\leq c\,(t\wedge 1)^{-\rho}\,\sup_{x \in\,H}|\varphi(x)|,\]
for some constant $\rho>0$ (see \cite[Proposition 4.4.3 and Theorem 7.3.1]{libro}). If, as in the present paper, we do not want to assume those restrictive assumptions  on the non-degeneracy of the noise, the strong Feller property does not hold anymore and it has to be replaced by the so-called {\em asymptotic strong Feller property}.  As a matter of fact, thanks to condition \eqref{ghc411}, it is known that for every $\varphi \in\,C_b^1(H)$ and $t\geq 0$
\begin{equation}
\label{ghc410}
|D(P_t\varphi)(x)|_H\leq c\,\sqrt{P_t|\varphi|^2(x)}+\a(t)\sqrt{P_t|D\varphi|^2_H(x)},\ \ \ \ x \in\,H,
\end{equation}
for some function $\a(t)$ such that $\a(t)\to 0$, as $t\to\infty$. In Appendix \ref{appa} we give a proof of this bound for the reader's convenience.

Moreover, in \cite{libro} it is proven that the solution $u^x(t)$ of equation \eqref{eq1-par} is differentiable with respect to the initial condition $x \in\,H$ and if $D_x u^x(t)h$ denotes its derivative along the direction $h \in\,H$, it holds
\begin{equation}
\label{ghc412}
\sup_{x \in\,H}|D_x u^x(t)h|_H\leq c(t)\,|h|_H,\ \ \ \ \mathbb{P}-\text{a.s.}
\end{equation}
for some continuous increasing function $c(t)$.

\section{The main result}
\label{sec5}
In the previous section, we have seen  that there exists a sequence of probability measures $\{\nu_\mu\}_{\mu>0}\subset \mathcal{P}(\mathcal{H}_1)$ which are invariant for the semigroup $P^\mu_t$. The family  $\{\Pi_1\nu_\mu\}_{\mu>0}\subset \mathcal{P}(H^1)$ can be extended to the family $\{(\Pi_1\nu_\mu)^\prime\}_{\mu>0}\subset\mathcal{P}(H)$, by setting  every Borel set $A\subset H$
\[(\Pi_1\nu_\mu)^\prime(A)=\Pi_1\nu_\mu(A\cap H^1),\ \ \ \ \mu>0.\]

By using the results proven in  \cite{HM}, it is possible to show that  in the case we are considering  conditions \eqref{ghc400} and \eqref{ghc400bis} and the gradient estimate \eqref{ghc412}, together with \eqref{ghc410}, imply that there exist a distance $\a$ on $H$ and  some constant $\delta>0$ such that
\begin{equation}
\label{ghc413}
\rho_\a\le(P_t^\star \nu_1,P_t^\star \nu_2\r)\leq c\,e^{-\delta t}\rho_\a(\nu_1,\nu_2),\ \ \ \ t\geq 0,
\end{equation}
for every $\nu_1, \nu_2 \in\,\mathcal{P}r(H)$. Here
$P_t^\star \nu$ is the measure defined by 
\[\int_{H}\varphi(x)\,d P_t^\star \nu(x)=\int_{H}P_t \varphi(x)\,d\nu(x),\]
for every $\varphi \in\,C_b(H)$,
and
$\rho_\a$ is the Wasserstein distance in $\mathcal{P}(H)$ associated with $\a$, defined by
\[\rho_\a (\nu_1,\nu_2)=\inf_{[\varphi]_{\text{\tiny{Lip}}_{H,\a}\leq 1}}\le|\int_{H^1}\varphi(x)\,d\nu_1(x)-\int_{H^1}\varphi(x)\,d\nu_2(x)\r|,\]
where
\[[\varphi]_{\text{\tiny{Lip}}_{H,\a}}=\sup_{x\neq y}\frac{|\varphi(x)-\varphi(y)|}{\a(x,y)}.\]
Notice that it can be proven that
\begin{equation}
\label{ghc415}
\rho_\a(\nu_1,\nu_2)=\inf\,\mathbb{E}\,\a(X,Y),\end{equation}
over all random variables $X$ and $Y$ distributed as $\nu_1$ and $\nu_2$, respectively (for a proof see \cite{villani}).

\medskip

The choice of a distance $\a$ on $H$ which allows to obtain the contraction property \eqref{ghc413} strongly depends on the type of non-linearity $b$ we are dealing with. 

In the case the constant $\la$ in Hypothesis \ref{H3} belongs to $(1,3]$, the uniform  condition \eqref{ghc400bis} holds. By proceeding as in \cite[Lemma 4.1]{fghr} it is possible to show that \eqref{ghc400bis} implies that for every $\e>0$ there exists $t_\star=t_\star(\e)>0$ such that 
\[\inf_{x, y \in\,H}\sup_{\Gamma \in\,\mathcal{C}(P_t^\star \d_x, P_t^\star \d_y)}\Gamma\le((x^\prime,y^\prime) \in\,H\times H\ :\ |x^\prime-y^\prime|_H<\e\r)>0,\ \ \ \ t\geq t_\star,\]
where $\mathcal{C}(P_t^\star \d_x, P_t^\star \d_y)$ denotes the collection of all couplings of the measures $P_t^\star \d_x$ and $P_t^\star \d_y$.
This means that \cite[Assumption 3]{HM} is satisfied. Moreover, since we have \eqref{ghc410}, also \cite[Assumption 2]{HM} is satisfied.
Therefore, thanks to \cite[Theorem 2.5]{HM} we have that \eqref{ghc413} holds for 
\begin{equation}
\label{fine337}
\a(x,y)=|x-y|_H.\end{equation}

In the case $\la=1$, bound \eqref{ghc400}  is uniform only on bounded sets of $H$. Hence a Liapunov structure has to be taken into the picture and, as proven in \cite[Theorem 3.4]{HM}, condition \eqref{ghc413} is true with the following distance $\a$
\begin{equation}
\label{f2001}
\a(x,y)=\inf_{\gamma}\int_0^1 \exp\le(\eta\,|\gamma(t)|^2_H\r)\,|\gamma^\prime(t)|_H\,dt,\end{equation}
 where the infimum is taken over all paths $\gamma$ such that $\gamma(0)=x$ and $\gamma(1)=y$ and where
 $\eta$ is some positive constant, which is less than a value $\eta_0$ determined by $\a_1-L_b$. Notice that for every $x, y \in\,H$
\begin{equation}
\label{f2000}
|x-y|_H\leq \a(x,y)\leq \exp\le(\eta(|x|_H^2+|y|_H^2)\r)|x-y|_H
\end{equation}

By proceeding again as in \cite[Lemma 4.1]{fghr}, it is possible to show that \eqref{ghc400} implies that for every $M, \e>0$ there exists $t_\star=t_\star(M,\e)>0$ such that 
\[\inf_{x, y \in\,B_H(M)}\sup_{\Gamma \in\,\mathcal{C}(P_t^\star \d_x, P_t^\star \d_y)}\Gamma\le((x^\prime,y^\prime) \in\,H\times H\ :\ |x^\prime-y^\prime|_H<\e\r)>0,\ \ \ \ t\geq t_\star.\]
This means that \cite[Assumption 6]{HM} is verified. Moreover, if we define $U(t)=\eta\,|u(t)|_H^2$, from It\^o's formula we have
\[dU(t)=\eta\le[\mbox{Tr} Q^2-|u(t)|_{H^1}^2+\langle B(u(t)),u(t)\rangle_H\r]\,dt+2\eta\,|Qu(t)|_H\,d\beta(t),\]
for some Brownian motion $\beta(t)$. It is then immediate to check that we can apply \cite[Lemma 5.1]{HM}, so that \cite[Assumption 4]{HM} holds for $V(x)=\exp(\eta\,|x|_H^2)$.
Finally, due to \eqref{ghc410}, \cite[Assumption 5]{HM} holds. This means that we can apply \cite[Theorem 3.4]{HM} and conclude that \eqref{ghc413} holds for $\a$ defined as in \eqref{f2001}.

\medskip

In what follows, we shall prove that, if $\nu$ is the unique invariant measure of equation \eqref{eq1-par}, then the following result holds.
\begin{Theorem}
\label{main}
Under Hypotheses \ref{H1} and  \ref{H3}, we have
\begin{equation}
\label{ghc414}
\lim_{\mu\to 0}\rho_\a \le((\Pi_1\nu_\mu)^\prime,\nu\r)=0,
\end{equation}
where $\a$ is either \eqref{fine337} or \eqref{f2001}, in the case $\la \in\,(1,3]$ and $\la=1$, respectively.
\end{Theorem}
\begin{proof}
The method used here is analogous to the one used for example in \cite{fghr}.
Due to the invariance of $\nu_\mu$ and $\nu$, for every $t\geq 0$ we have
\[\begin{array}{l}
\ds{\rho_\a\le((\Pi_1\nu_\mu)^\prime,\nu\r)=\rho_\a\le((\Pi_1 (P_t^\mu)^\star \nu_\mu)^\prime,P_t^\star\nu\r)}\\
\vs
\ds{\leq 
\rho_\a\le((\Pi_1(P_t^\mu)^\star \nu_\mu)^\prime,P_t^\star(\Pi_1\nu_\mu)^\prime\r)+\rho_\a\le(P_t^\star(\Pi_1\nu_\mu)^\prime,P_t^\star\nu\r).}
\end{array}\]
Hence,
thanks to \eqref{ghc413}, we obtain
\[\rho_\a\le((\Pi_1\nu_\mu)^\prime,\nu\r)\leq \rho_\a\le((\Pi_1(P_t^\mu)^\star \nu_\mu)^\prime,P_t^\star(\Pi_1\nu_\mu)^\prime\r)+c\,e^{-\delta t}\rho_\a\le((\Pi_1\nu_\mu)^\prime,\nu\r).\]
This means that, if we pick $t_\star>0$ such that $c\,e^{-\delta t_\star}\leq 1/2$, we have
\[\rho_\a\le((\Pi_1\nu_\mu)^\prime,\nu\r)\leq 2\,\rho_\a\le((\Pi_1(P_t^\mu)^\star \nu_\mu)^\prime,P_t^\star(\Pi_1\nu_\mu)^\prime\r),\ \ \ \ t\geq t_\star.\]
In particular, \eqref{ghc414} follows once we are able to show that
\[\lim_{\mu\to 0}\rho_\a\le((\Pi_1(P_t^\mu)^\star \nu_\mu)^\prime,P_t^\star(\Pi_1\nu_\mu)^\prime\r)=0.
\]

Due to \eqref{ghc415}, we have that
\[\rho_\a\le((\Pi_1(P_t^\mu)^\star \nu_\mu)^\prime,P_t^\star(\Pi_1\nu_\mu)^\prime\r)\leq \mathbb{E}\,\a(u^{\gamma_\mu}(t),u^{\theta_\mu}(t)),\]
where $\gamma_\mu=(\theta_\mu,\eta_\mu)$ is any $\mathcal{H}_1$-valued random variable such that $\mathcal{L}(\gamma_\mu)=\nu_\mu$. Therefore, \eqref{ghc414} follows once we show that
\begin{equation}
\label{ghc416}
\lim_{\mu\to 0}\mathbb{E}\,\a(u^{\gamma_\mu}(t),u^{\theta_\mu}(t))=0.
\end{equation}

\end{proof}

The proof of limit \eqref{ghc416}, which concludes the proof of Theorem  \ref{main}, is given in the next section.

\section{Proof of limit \eqref{ghc416}}
\label{sec6}

In Section \ref{sec3}, we have seen that for every $\mu>0$ equation \eqref{eq1-abstr} admits an invariant measure. Let us denote by  $\{\nu_\mu\}_{\mu>0}$ any family of invariant measures for equation \eqref{eq1}.

\begin{Theorem}
\label{fund}
Assume Hypotheses \ref{H1} and  \ref{H3}. For every $\mu>0$, let $\gamma_\mu=(\theta_\mu,\eta_\mu)$ be a $\mathcal{F}_0$-measurable $\mathcal{H}_1$-valued random variable, distributed as the invariant measure $\nu_\mu$. Then, for every $T>0$ we have
\begin{equation}
\label{ghc500}
\lim_{\mu\to 0}\sup_{t \in\,[0,T]}\,\mathbb{E}\,\a(u^{\gamma_\mu}_\mu(t),u^{\theta_\mu}(t))=0.
\end{equation}
\end{Theorem}

We will prove the theorem above by distinguishing the case $b(\xi,\cdot)$ is Lipschitz-continuous (that is $\la=1$) and the case $b(\xi,\cdot)$ has polynomial growth (that is $\la \in\,(1,3]$).
In both cases the following fundamental lemma, whose proof is postponed to  Section \ref{sec7}, holds.

\begin{Lemma}
\label{l5.2}
Assume Hypotheses \ref{H1} and  \ref{H3} holds and fix any $T>0$. 
\begin{enumerate}
\item If $\la=1$, then for every $0<\beta\leq 1$ and  $\gamma \in\,L^2(\Omega;\mathcal{H}_\beta)$
it holds
\begin{equation}
\label{ghc514}
\sup_{\mu \in\,(0,1]}\mathbb{E}\,\sup_{t \in\,[0,T]}|u^\gamma_\mu(t)|_{H^\beta}^2\leq c_T\le(1+\mathbb{E}\,|\gamma|^2_{H^\beta}\r).\end{equation}
\item If $\la \in\,(1,3]$, then for any  $\gamma=(\theta,\eta) \in\,L^2(\Omega;\mathcal{H}_1)$, such that $\theta \in\,L^{\la+1}(\Omega;L^{\la+1}(D))$, and for any $R>0$, it holds
\begin{equation}
\label{ghc300}
\sup_{\mu \in\,(0,1]} \mathbb{P}\le(\sup_{t \in\,[0,T]}|u_\mu^\gamma(t)|_{H^1}\geq R\r)\leq c_T(R)\le(1+\mathbb{E}\,|\gamma|^2_{\mathcal{H}_1}+\mathbb{E}\,|\theta|_{L^{\la+1}}^{\la+1}\r),
\end{equation}
for some function $c_T:[0,+\infty)\to[0,+\infty)$ such that
\[\lim_{R\to\infty} c_T(R)=0.\]
\end{enumerate}
\end{Lemma}

\subsection{Proof of Theorem \ref{fund} in the Lipschitz case}

We recall that, due to \eqref{f2000}, we have 
\[\a(x,y)\leq \exp\le(\eta\le[|x|_H^2+|y|_H^2\r]\r)|x-y|_H,\ \ \ \ \ x, y \in\,H.\]
For every $\e>0$ we have
\[\begin{array}{l}
\ds{\exp\le(\eta\le[|x|_H^2+|y|_H^2\r]\r)|x-y|_H\leq c\,\exp\le(\eta\le[|x|_H^2+|y|_H^2\r]\r)\le(|x|_H^2+|y|_H^2+1\r)|x-y|_H^{1/2}}\\
\vs
\ds{\leq c_\e\exp\le((\eta+\e)\,\le[|x|_H^2+|y|_H^2\r]\r)|x-y|_H^{1/2},}
\end{array}\]
for some $c_\e>0$.
This implies that
\[\begin{array}{l}
\ds{\mathbb{E}\,\a(u_\mu^{\gamma_\mu}(t),u^{\theta_\mu}(t)) }\\
\vs
\ds{\leq c_\e\le(\mathbb{E}\,\exp\le(2(\eta+\e)\,\le[|u_\mu^{\gamma_\mu}(t)|_H^2+|u^{\theta_\mu}(t)|_H^2\r]\r)\r)^{1/2}\le(\mathbb{E}\,|u_\mu^{\gamma_\mu}(t)-u^{\theta_\mu}(t)|_H\r)^{1/2}.}
\end{array}\]
Therefore, \eqref{ghc500} holds if we can prove that 
\begin{equation}
\label{ghc540}
\lim_{\mu\to 0}\,\mathbb{E}\le(|u_\mu^{\gamma_\mu}(t)-u^{\theta_\mu}(t)|_H\r)=0,
\end{equation}
and, for some $\mu_0>0$ and $\bar{\eta}>2 \eta>0$,
\begin{equation}
\label{h10}
\sup_{\mu\leq \mu_0}\,\mathbb{E}\,\exp\le(\bar{\eta}      \,\le[|u_\mu^{\gamma_\mu}(t)|_H^2+|u^{\theta_\mu}(t)|_H^2\r]\r)<\infty.
\end{equation}

Concerning \eqref{h10}, due to Lemma \ref{lemma-c2} and Lemma \ref{lemma-c3} (cfr. \eqref{h7} and \eqref{h17}), there exists $\eta>0$ such that for every $T>0$ and $t\leq T$
\[\mathbb{E}\,\exp\le(\eta\,\le[|u_\mu^{\gamma_\mu}(t)|_H^2+|u^{\theta_\mu}(t)|_H^2\r]\r)\leq 
c_T\,\mathbb{E}\exp\le(\eta\,\le[\mu\,|\theta_\mu|_{H^1}^2+|\theta_\mu|_H^2+\mu^2\,|\eta_\mu|_H^2+\mu\,|\theta_\mu|_{L^{\la+1}}^{\la+1}\r]\r).\]
Hence, as $\mathcal{L}(\gamma_\mu)=\nu_\mu$, due to \eqref{h4} we get
\[\begin{array}{l}
\ds{\sup_{\mu\leq 1/2}\mathbb{E}\,\exp\le(\eta\,\le[|u_\mu^{\gamma_\mu}(t)|_H^2+|u^{\theta_\mu}(t)|_H^2\r]\r)}\\
\vs
\ds{\leq c_T\,\sup_{\mu\leq 1/2}\int_{\mathcal{H}_1}\exp\le(\eta\,\le[\mu\,|u|_{H^1}^2+|u|_H^2+\mu^2\,|v_\mu|_H^2+\mu\,|u|_{L^{\la+1}}^{\la+1}\r]\r)
\,d\nu_\mu(u,v)<\infty,}
\end{array}\]
and \eqref{h10} follows.

Now, let us prove \eqref{ghc540}. Referring back to \eqref{ghc513} and \eqref{fine335} and and using
 \eqref{s2} and the Lipschitz continuity of $B$ in $H$, we have
\[\begin{array}{l}
\ds{|u_\mu^{\gamma_\mu}(t)-u^{\theta_\mu}(t)|_H\leq |\Pi_1 S_\mu(t)\gamma_\mu-S(t)\theta_\mu|_H+c\int_0^t|u_\mu^{\gamma_\mu}(s)-u^{\theta_\mu}(s)|_H\,ds}\\
\vs
\ds{+\int_0^t\le|\frac 1\mu\Pi_1 S_\mu(t-s)(0,B(u^{\theta_\mu}(s))-S(t-s)B(u^{\theta_\mu}(s))\r|_H\,ds+|\Pi_1 \Gamma_\mu(t)-\Gamma(t)|_H.}
\end{array}\]
Therefore, as a consequence of the Gronwall lemma, we have
\[\begin{array}{l}
\ds{\sup_{t \in\,[0,T]}\,|u_\mu^{\gamma_\mu}(t)-u^{\theta_\mu}(t)|_H\leq c_T\,\sup_{t \in\,[0,T]}\,|\Pi_1 S_\mu(t)\gamma_\mu-S(t)\theta_\mu|_H}\\
\vs
\ds{+c_T\sup_{t \in\,[0,T]}\int_0^t\le|\frac 1\mu\Pi_1 S_\mu(t-s)(0,B(u^{\theta_\mu}(s))-S(t-s)B(u^{\theta_\mu}(s))\r|_H\,ds}\\
\vs
\ds{+c_T\,\sup_{t \in\,[0,T]}\,|\Pi_1 \Gamma_\mu(t)-\Gamma(t)|_H=:c_T \sum_{k=1}^3 I_{\mu,k}.}
\end{array}\]
In \cite[proof of Theorem 4.1]{cs}, it is shown that 
\[\lim_{\mu\to 0} \mathbb{E} I_{\mu,3}=0.\]
Therefore, in order to get \eqref{ghc540}, we have to prove  
\[\lim_{\mu\to 0}\mathbb{E}\,I_{\mu,k}=0,\ \ \ \ k=1,2.\]

\medskip

We start with $I_{\mu,1}$. For every fixed $R>0$ and $\beta \in\,(0,1]$, we have
\[\begin{array}{l}
\ds{\mathbb{E}\,I_{\mu,1}\leq \mathbb{E}\le(I_{\mu,1} ; |\theta_\mu|_{H^\beta}+\sqrt{\mu}\,|\eta_\mu|_{H}\leq R\r)}\\
\vs
\ds{+\mathbb{E}\le(I_{\mu,1} ; \{|\theta_\mu|_{H^\beta}\geq  R/2\}\cup \{\sqrt{\mu} |\eta_\mu|_{H}\geq R/2\}\r)=:J^{\mu,1}_R+J^{\mu,2}_R.}
\end{array}\]
Due to \eqref{energy-est-1}, we have
\[\begin{array}{l}
\ds{J^{\mu,2}_R\leq \frac c{R}\,\le(\mathbb{E}|I_{\mu,1}|^2\r)^{\frac 12} \le(\mathbb{E}\,|\theta_\mu|_{H^\beta}^2+\mu\,\mathbb{E}|\eta_\mu|^2_{H}\r)^{\frac 12}}\\
\vs
\ds{\leq \frac{c_T}{R}\le(\mathbb{E}|\theta_\mu|_H^2+\mu\,\mathbb{E}|\eta_\mu|_{H^{-1}}^2\r)^{\frac 12}\le(\mathbb{E}\,|\theta_\mu|_{H^\beta}^2+\mu\,\mathbb{E}|\eta_\mu|^2_{H}\r)^{\frac 12}\leq  \frac{c_T}{R}\le(\mathbb{E}\,|\theta_\mu|_{H^1}^2+\mu\,\mathbb{E}|\eta_\mu|^2_{H}\r)}\\
\vs
\ds{=\frac{c_T}{R}\int_{\mathcal{H}_1}\le(|u|_{H^1}^2+\mu\,|v|_H^2\r)\,d\nu_\mu(u,v).}
\end{array}\]
Therefore, thanks to \eqref{ghc100}, we have
\[\sup_{\mu \in\,(0,1/2]}J^{\mu,2}_R\leq \frac {c_T}R.\]
This means that if we fix any $\e>0$ we can find $R_\e>0$ such that
$J^{\mu,2}_{R_\e}\leq \e$, for every $\mu\in\,(0,1/2]$.
Once fixed $R_\e$, due to \eqref{Smu-to-Se-at-a-point}, we have
\[\limsup_{\mu\to 0}\mathbb{E} I_{\mu,1}\leq \e+\lim_{\mu\to 0}J^{\mu,1}_{R_\e}=\e,\] so that, due to the arbitrariness of $\e$, we conclude that
\begin{equation}
\label{ghc510}
\lim_{\mu\to 0}\,\sup_{t \in\,[0,T]}\,|\Pi_1 S_\mu(t)\gamma_\mu-S(t)\theta_\mu|_H=\lim_{\mu\to 0}\,\mathbb{E} I_{\mu,1}=0.
\end{equation}

\medskip

Concerning $I_{\mu,2}$, 
we have, for any $R>0$, 
\[\begin{array}{l}
\ds{\mathbb{E}\,I_{\mu,2}=\mathbb{E}\le(I_{\mu,2};\,|B(u^{\theta_\mu}(\cdot))|_{L^1(0,T;H^\beta)}\leq R\r)+\mathbb{E}\le(I_{\mu,2};\,|B(u^{\theta_\mu}(\cdot))|_{L^1(0,T;H^\beta)}> R\r)}\\
\vs
\ds{=:J^{\mu,1}_R+J^{\mu,2}_R.}
\end{array}\]

As a consequence of \eqref{ghc512},  \eqref{ghc513} and \eqref{ghc520}
\[\begin{array}{l}
\ds{\mathbb{E}\,|I_{\mu,2}|^2\leq c_{\beta,T}\,\mathbb{E}\,|B(u^{\theta_\mu})|^2_{L^2(0,T;H^\beta)}\leq c_{\beta,T}\,\le(1+\mathbb{E}\,|u^{\theta_\mu}|^2_{L^2(0,T;H^\beta)}\r)\leq c_{\beta,T}\le(1+\mathbb{E}\,|\theta_\mu|_{H^\beta}^2\r).}
\end{array}\]
Thanks to \eqref{ghc100}, this implies that
\[
\sup_{\mu \in\,(0,1/2]}\,\mathbb{E}\,|I_{\mu,2}|^2\leq c_{\beta,T}\,\sup_{\mu \in\,(0,1/2]}\,\mathbb{E}\,|B(u^{\theta_\mu})|^2_{L^2(0,T;H^\beta)}<+\infty.
\]
Therefore, for every $\e>0$ we can find $R_\e>0$ such that 
\[\sup_{\mu \in\,(0,1/2]} \,J^{\mu,2}_{R_\e}\leq \frac c{R_\e}\leq \e.\]
Once fixed $R_\e$, due to \eqref{Smu-to-Se-2} we have
\[\lim_{\mu\to 0}J^{\mu,1}_{R_\e}=0,\] so that we conclude that
\[\lim_{\mu\to 0}\mathbb{E}\,I_{\mu,2}=0.\]
This, together with \eqref{ghc510}, allows to obtain \eqref{ghc540}.

\subsection{Proof of Theorem \ref{fund} in the non Lipschitz case}

In this case
$\a(x,y)=|x-y|_H$, so that \eqref{ghc500} follows if we prove that 
\[\lim_{\mu\to 0} \mathbb{E}\,|u^{\gamma_\mu}_\mu(t)-u^{\theta_\mu}(t)|_H=0.\]
For every $R>0$, we have
\begin{equation}
\label{ghc30}
\begin{array}{l}
\ds{\mathbb{E}\,|u^{\gamma_\mu}_\mu(t)-u^{\theta_\mu}(t)|_{H}\leq \mathbb{E}\le(|u^{\gamma_\mu}_\mu(t)-u^{\theta_\mu}(t)|_{H}\,;\,\tau^\mu_R\leq t\r)+\mathbb{E}\le(|u^{\gamma_\mu}_\mu(t)-u^{\theta_\mu}(t)|_{H}\,;\,\tau^\mu_R> t\r)}\\
\vs
\ds{=:J^{\mu,1}_R(t)+J^{\mu,2}_R(t),}
\end{array}
\end{equation}
where
\[\tau^\mu_R(t):=\inf\le\{ s\geq 0\,:\,\max\,\le(|u^{\gamma_\mu}_\mu(s)|_{L^\infty}, |u^{\theta_\mu}(s)|_{L^\infty}\r)\geq R \r\}.\]

Concerning the first term $J^{\mu,1}_R(t)$, we have
\[\begin{array}{l}
\ds{|J^{\mu,1}_R(t)|^2\leq \mathbb{E} \,|u^{\gamma_\mu}_\mu(t)-u^{\theta_\mu}(t)|^2_{H}\,\mathbb{P}\le(\tau^\mu_R\leq t\r)}\\
\vs
\ds{\leq 2\le[\mathbb{E}|u^{\gamma_\mu}_\mu(t)|_H^2+\mathbb{E}|u^{\theta_\mu}(t)|^2_{H}\r]\le[\mathbb{P}\le(\sup_{s\leq t}|u^{\gamma_\mu}_\mu(s)|_{L^\infty}>R\r)+\mathbb{P}\le(\sup_{s\leq t}|u^{\theta_\mu}(s)|_{L^\infty}>R\r)\r]}
\end{array}\]

Thanks to \eqref{ghc1}, for every $\mu\leq 1$ we get
\[\sup_{t \in\,[0,T]}\,\mathbb{E}|u^{\gamma_\mu}_\mu(t)|_H^2\leq c_T\le(1+\mu\,\mathbb{E}\,|\theta_\mu|^2_{H^1}+\mathbb{E}\,|\theta_\mu|_H^2+\mu^2 \mathbb{E}\,|\eta_\mu|_H^2+\,\mu \,\mathbb{E}\,|\theta_\mu|_{L^{\la+1}}^{\la+1}\r).\]
Moreover, due to \eqref{ghc533} we have
\[\mathbb{E}\,\sup_{t \in\,[0,T]}|u^{\theta_\mu}(t)|_H^2\leq c_T\le(1+\mathbb{E}\,|\theta_\mu|_H^2\r).\]
Therefore, in view of \eqref{ghc100}, we can conclude that
for every $R>0$ and $\mu \in\,(0,1/2]$
\[\sup_{t \in\,[0,T]}|J^{\mu,1}_R(t)|^2\leq c_T\,\sup_{t \in\,[0,T]}\,\le[\mathbb{P}\le(\sup_{s\leq t}|u^{\gamma_\mu}_\mu(s)|_{L^\infty}>R\r)+\mathbb{P}\le(\sup_{s\leq t}|u(s)|_{L^\infty}>R\r)\r].\]

As shown in Lemma \ref{l5.2}, for every $R>0$ and $\mu \in\,(0,1]$
\[\mathbb{P}\le(\sup_{s\leq t}|u^{\gamma_\mu}_\mu(s)|_{L^\infty}>R\r)\leq c_T(R)\le(1+\mathbb{E}\,|\gamma_\mu|^2_{\mathcal{H}_1}+\mathbb{E}\,|\theta_\mu|_{L^{\la+1}}^{\la+1}\r).\]
Analogously, as a consequence of \eqref{ghc521}, for  every $\mu>0$ we have
\[\mathbb{P}\le(\sup_{s\leq t}|u^{\theta_\mu}(s)|_{L^\infty}>R\r)\leq \frac{c_T}{R^2}\le(1+\mathbb{E}\,|\theta_\mu|_{H^1}^2\r).\]
Therefore, since $c_T(R)\to 0$, as $R\to \infty$, thanks to \eqref{ghc100}, for any $\e>0$ there exists  $R_\e>0$ such that
\[\sup_{\mu \in\,(0,1/2]}\,\sup_{t \in\,[0,T]}\,J^{\mu,1}_{R_\e}(t)\leq \e.
\]
Thanks to \eqref{ghc30}, this yields
\begin{equation}
\label{ghc102}
\sup_{t \in\,[0,T]}\,\mathbb{E}|u^{\gamma_\mu}_\mu(t)-u^{\theta_\mu}(t)|_H\leq \e+J^{\mu,2}_{R_\e}(t),\ \ \ \ \mu \in\,(0,1/2].
\end{equation}

Now, let us estimate $J_{R_\e}^{\mu,2}(t)$. For every $\mu>0$, we have
\[\begin{array}{l}
\ds{u^{\gamma_\mu}_\mu(t)-u^{\theta_\mu}(t)=\le[\Pi_1 S_\mu(t)(\theta_\mu,\eta_\mu)-S(t)\theta_\mu\r]}\\
\vs
\ds{+\int_0^t\le[\Pi_1 S_\mu(t-s)B_\mu(z^{\gamma_\mu}_\mu(s))-S(t-s)B(u^{\theta_\mu}(s))\r]\,ds+\le[\Pi_1  \Gamma_\mu(t)-\Gamma(t)\r]=:\sum_{i=1}^3 I^{\mu}_{i}(t),}
\end{array}\]
where $\Gamma_\mu(t)$ and $\Gamma(t)$ are the two stochastic convolutions, defined respectively in \eqref{ghc3} and \eqref{ghc4}.
This implies that 
\begin{equation}
\label{ghc31}
\begin{array}{l}
\ds{J^{\mu,2}_{R_\e}(t)\leq \mathbb{E}\,|I^\mu_1(t)|_H+\mathbb{E}|I^\mu_3(t)|_H+\mathbb{E}\le(\le|I^\mu_2(t)\r|_{H}\,;\,\tau^\mu_{R_\e}> t\r).}
\end{array}
\end{equation}

We have
\[\begin{array}{l}
\ds{I^\mu_2(t)=\int_0^t\Pi_1 S_\mu(t-s)\le[B_\mu(z^{\gamma_\mu}_\mu(s))-B_\mu(u^{\theta_\mu}(s),0)\r]\,ds+\int_0^t\Phi_\mu(t-s)B(u^{\theta_\mu}(s))\,ds}\\
\vs
\ds{=:I^{\mu,1}_2(t)+I^{\mu,2}_2(t),}
\end{array}\]
where $\Phi_\mu(t)$ is the operator introduced in \eqref{ghc534}.

According to \eqref{s2}, we have, in view of 
\[|I^{\mu,1}_2(t)|_{H}\leq 2 \int_0^t|B(u^{\gamma_\mu}_\mu(s))-B(u^{\theta_\mu}(s))|_H\,ds.\]
For every $u,v \in\,L^\infty(D)$, we have, in view of \eqref{sk1}, that 
\[\begin{array}{l}
\ds{|B(u)-B(v)|_H\leq c\,\le||u-v|\le(1+|u|^{\la-1}+|v|^{\la-1}\r)\r|_H}\\
\vs
\ds{\leq c\,|u-v|_H\le(1+|u|_{L^\infty}^{\la-1}+|v|_{L^\infty}^{\la-1}\r).}
\end{array}\]
Then
\begin{equation}
\label{ghc104}
\begin{array}{l}
\ds{\mathbb{E}\le(|I^{\mu,1}_2(t)|_{H}\,;\,\tau^\mu_{R_\e}> t\r)\leq c\,\le(1+2\,R_\e^{\la-1}\r)\int_0^t\mathbb{E}\le(|u^{\gamma_\mu}_\mu(s)-u(s)|_H\,;\,\tau^\mu_{R_\e}>t\r)\,ds}\\
\vs
\ds{\leq c\,\le(1+2\,R_\e^{\la-1}\r)\int_0^t J^{\mu,2}_{R_\e}(s)\,ds.}
\end{array}
\end{equation}
Combining this bounds with  to \eqref{ghc31}, this yields
\[\begin{array}{l}
\ds{J^{\mu,2}_{R_\e}(t)\leq \mathbb{E}\,|I^\mu_1(t)|_H+\mathbb{E}|I^\mu_3(t)|_H+\mathbb{E}\le(|I^{\mu,2}_2(t)|_{H}\,;\,\tau^\mu_{R_\e}> t\r)+c\,\le(1+2\,R_\e^{\la-1}\r)\int_0^t J^{\mu,2}_{R_\e}(s)\,ds,}
\end{array}\]
and, due to the Gronwall lemma, for every $t \in\,[0,T]$
\begin{equation}
\label{ghc110}
J^{\mu,2}_{R_\e}(t)\leq c_{T,\e}\sup_{t \in\,[0,T]}\,\le( \mathbb{E}\,|I^\mu_1(t)|_H+\mathbb{E}|I^\mu_3(t)|_H+\mathbb{E}\le(|I^{\mu,2}_2(t)|_{H}\,;\,\tau^\mu_{R_\e}> t\r)\r).
\end{equation}

Next, if $\tau^\mu_{R_\e}>t$,  in view of \eqref{fine339} we have
\[|B(u^{\theta_\mu}(s))|_{H^1}\leq c\le(1+|u^{\theta_\mu}(s)|_{L^\infty}^{\la-1}\r)|u^{\theta_\mu}(s)|_{H^1}\leq c\,\le(1+R_\e^{\la-1}\r)|u^{\theta_\mu}(s)|_{H^1},\ \ \ \ s\leq t,\]
so that, thanks to Lemma \ref{a-priori1},
\[\le|B(u^{\theta_\mu}(\cdot)) \mathbb{I}_{\{\tau^\mu_{R_\e}>t\}}\r|_{L^2([0,t];H^1)}\leq c,\ \ \ \ \mathbb{P}-\text{a.s.}\]
Thanks to \eqref{Smu-to-Se-2}, this allows us to conclude that \begin{equation}
\label{ghc105}
\lim_{\mu\to 0} \mathbb{E}\,\sup_{t \in\,[0,T]}\,\le(\le|I^{\mu,2}_2(t)\r|_{H}\,;\,\tau^\mu_{R_\e}> t\r)=0.\end{equation}

Moreover, in \cite[Theorem 4.1]{cs} it is proven that
\begin{equation}
\label{ghc108}
\lim_{\mu\to 0}\sup_{ t \in\,[0,T]}\mathbb{E}\,|I^\mu_3(t)|_H=0.\end{equation}
Therefore, collecting together \eqref{ghc110}, \eqref{ghc105} and \eqref{ghc108}, in view of \eqref{ghc510}, we conclude
\[\lim_{\mu\to 0}
J^{\mu,2}_{R_\e}(t)=0.\]
According to \eqref{ghc102}, due to the arbitrariness of $\e>0$, \eqref{ghc500} follows.

\section{Proof of Lemma \ref{l5.2}}
\label{sec7}

In our proof of Lemma \ref{l5.2}, we distinguish the case $\la=1$ and the case $\la \in\,(1,3]$.

\medskip

{\em Case $\la=1$.} We have 
\[u_\mu^\gamma(t)=\Pi_1 S_\mu(t)\gamma+\int_0^t \Pi_1 S_\mu(t-s) B_\mu(z_\mu^\gamma(s))\,ds+\Pi_1 \Gamma_\mu(t).\]
Then, thanks to \eqref{s2} and \eqref{energy-est-1}, for every $t \in\,[0,T]$ and $\mu \in\,(0,1]$ we have
\[|u_\mu^\gamma(t)|^2_{H^\beta}\leq c\,|\gamma|^2_{\mathcal{H}_\beta}+c_T\,\int_0^t|B(u^\gamma_\mu(s))|^2_{H^\beta}\,ds+c\,\le|\Pi_1 \Gamma_\mu(t)\r|^2_{H^\beta}.\]
Due to \eqref{ghc513}, this implies
\[|u_\mu^\gamma(t)|^2_{H^\beta}\leq c_T\,\le(|\gamma|^2_{\mathcal{H}_\beta}+\sup_{t \in\,[0,T]}\,\le|\Pi_1 \Gamma_\mu(t)\r|^2_{H^\beta}\r)+c_T\int_0^t\le(|u^{\gamma}_\mu(s)|_{H^\beta}^2+1\r)\,ds,\]
so that, thanks to the Gronwall lemma and \eqref{ghc303} we can conclude that
\eqref{ghc514} holds.

{\em Case $\la \in\,(1,3]$.} If we denote
\[v_\mu^\gamma(t):=u_\mu^\gamma(t)-\Pi_1\Gamma_\mu(t),\]
we have that $v_\mu^\gamma$ solves the equation
\begin{equation}
\label{ghc305}
\le\{\begin{array}{l}
\ds{\mu\, \partial^2_t v^\gamma_\mu(t,\xi)=\Delta v^\gamma_\mu(t,\xi)-\partial_t v^\gamma_\mu(t,\xi)+b(\xi,u_\mu(t,\xi)),}\\
\vs
\ds{v^\gamma_\mu(0,\xi)=\theta(\xi),\ \ \partial_t v^\gamma_\mu(0,\xi)=\eta(\xi),\ \ \xi \in\,D,\ \ \ \ v^\gamma_\mu(t,\xi)=0,\ \ \ \ t\geq 0,\ \ \xi \in\,\partial D.}
\end{array}
\r.
\end{equation}
This implies
\[\begin{array}{l}
\ds{\mu\,\frac d{dt}|\partial_t v_\mu^\gamma(t)|_{H}^2+\frac d{dt}|v_\mu^\gamma(t)|_{H^1}^2+2\,|\partial_t v_\mu^\gamma(t)|_{H}^2}\\
\vs
\ds{=2\,\langle b(\cdot,v^\gamma_\mu(t)),\partial_t v^\gamma_\mu(t)\rangle_H+2\,\langle b(\cdot,v^\gamma_\mu(t)+\Pi_1\Gamma_\mu(t))-b(\cdot,v^\gamma_\mu(t)),\partial_t v^\gamma_\mu(t)\rangle_H}\\
\vs
\ds{\leq 2\,\frac d{dt}\int_D \mathfrak{b}(\xi,v_\mu^\gamma(t,\xi))\,d\xi+\le|b(\cdot,v^\gamma_\mu(t)+\Pi_1\Gamma_\mu(t))-b(\cdot,v^\gamma_\mu(t))\r|_H^2+|\partial_t v^\gamma_\mu(t)|_H^2,}
\end{array}\]
where $\mathfrak{b}$ is the antiderivative of $b$ that satisfies \eqref{ghc140}. Thus
\begin{equation}
\label{ghc307}
\begin{array}{l}
\ds{\mu\,\frac d{dt}|\partial_t v_\mu^\gamma(t)|_{H}^2+\frac d{dt}|v_\mu^\gamma(t)|_{H^1}^2+|\partial_t v_\mu^\gamma(t)|_{H}^2-2 \frac d{dt}\int_D \mathfrak{b}(\xi,v_\mu^\gamma(t,\xi))\,d\xi}\\
\vs
\ds{\leq \le|b(\cdot,v^\gamma_\mu(t)+\Pi_1\Gamma_\mu(t))-b(\cdot,v^\gamma_\mu(t))\r|_H^2.}
\end{array}
\end{equation}
Now,  for every $x \in\,L^{2(\la-1)}(D)$ and $h \in\,L^\infty(D)$, we have
\[|b(x+h)-b(x)|_H\leq c\,\le||h|\le(1+|x|^{\la-1}+|h|^{\la-1}\r)\r|_H\leq c\,|h|_{L^\infty}\le(1+|x|_{L^{2(\la-1)}}^{\la-1}+|h|_{L^{2(\la-1)}}^{\la-1}\r).\]
Hence, as $H^1\hookrightarrow L^\infty(D)$ (we are assuming $d=1$ here), we have
\[\le|b(\cdot,v^\gamma_\mu(t)+\Pi_1\Gamma_\mu(t))-b(\cdot,v^\gamma_\mu(t))\r|_H^2\leq c\,\le(1+|\Pi_1\Gamma_\mu(t)|_{H^1}^{2\la}\r)+c\,|\Pi_1\Gamma_\mu(t)|_{H^1}^{2}|v^\gamma_\mu(t)|_{L^{2(\la-1)}}^{2(\la-1)}.\]
Since $\la\leq 3$, we have that $2(\la-1)\leq \la+1$, so that, thanks to \eqref{ghc140} we get
\[\begin{array}{l}
\ds{\le|b(\cdot,v^\gamma_\mu(t)+\Pi_1\Gamma_\mu(t))-b(\cdot,v^\gamma_\mu(t))\r|_H^2}\\
\vs
\ds{\leq c\,\le(1+|\Pi_1\Gamma_\mu(t)|_{H^1}^{2\la}\r)-c\,|\Pi_1\Gamma_\mu(t)|_{H^1}^{2}\int_D \mathfrak{b}(\xi,v^\gamma_\mu(t,\xi))\,d\xi.}
\end{array}\]
Therefore, if we replace the inequality above in \eqref{ghc307}, we obtain
\[\begin{array}{l}
\ds{\mu\,\frac d{dt}|\partial_t v_\mu^\gamma(t)|_{H}^2+\frac d{dt}|v_\mu^\gamma(t)|_{H^1}^2+|\partial_t v_\mu^\gamma(t)|_{H}^2-2 \frac d{dt}\int_D \mathfrak{b}(\xi,v_\mu^\gamma(t,\xi))\,d\xi}\\
\vs
\ds{\leq  c\,\le(1+|\Pi_1\Gamma_\mu(t)|_{H^1}^{2\la}\r)-c\,|\Pi_1\Gamma_\mu(t)|_{H^1}^{2}\int_D \mathfrak{b}(\xi,v^\gamma_\mu(t,\xi))\,d\xi.}
\end{array}\]

Now, if we integrate with respect to $t \in\,[0,T]$, we obtain
\[\begin{array}{l}
\ds{\mu\,|\partial_t v_\mu^\gamma(t)|_{H}^2+|v_\mu^\gamma(t)|_{H^1}^2+\int_0^t|\partial_t v_\mu^\gamma(s)|_{H}^2\,ds-2 \int_D \mathfrak{b}(\xi,v_\mu^\gamma(t,\xi))\,d\xi}\\
\vs
\ds{\leq  \mu\,|\eta|_{H}^2+|\theta|_{H^1}^2-2 \int_D \mathfrak{b}(\xi,\theta(\xi))\,d\xi+c\,\le(1+\sup_{t \in\,[0,T]}|\Pi_1\Gamma_\mu(t)|_{H^1}^{2\la}\r)}\\
\vs
\ds{-c\,\sup_{t \in\,[0,T]}\,|\Pi_1\Gamma_\mu(t)|_{H^1}^{2}\int_0^t\int_D \mathfrak{b}(\xi,v^\gamma_\mu(s,\xi))\,d\xi\,ds.}
\end{array}\]
Then, as a consequence of 
the Gronwall lemma and of \eqref{ghc140} we get
\[\begin{array}{l}
\ds{\sup_{t \in\,[0,T]}\le(\mu\,|\partial_t v_\mu^\gamma(t)|_{H}^2+|v_\mu^\gamma(t)|_{H^1}^2+|v_\mu^\gamma(t))|_{L^{\la+1}}^{\la+1}\r)+\int_0^T|\partial_t v_\mu^\gamma(s)|_{H}^2\,ds}\\
\vs
\ds{\leq  c_T\le(\mu\,|\eta|_{H}^2+|\theta|_{H^1}^2+|\theta|_{L^{\la+1}}^{\la+1}+c\,\le(1+\sup_{t \in\,[0,T]}|\Pi_1\Gamma_\mu(t)|_{H^1}^{2\la}\r)\r)\exp\le(T\,\sup_{t \in\,[0,T]}\,|\Pi_1\Gamma_\mu(t)|_{H^1}^{2}\r).}
\end{array}\]
This implies that, if we define
\[\La_\mu(T):=\sup_{t \in\,[0,T]}\,|\Pi_1\Gamma_\mu(t)|_{H^1}^{2},\]
for every $\mu \in\,(0,1]$ we have
\[\begin{array}{l}
\ds{\mathbb{P}\le(\sup_{t \in\,[0,T]}|v_\mu^\gamma(t)|_{H^1}^2\geq \frac{R^2}4\r)\leq \mathbb{P}\le(\exp (T \La_\mu(T))\geq \frac R2\r)}\\
\vs
\ds{+
\mathbb{P}\le(\mu\,|\eta|_{H}^2+|\theta|_{H^1}^2+|\theta|_{L^{\la+1}}^{\la+1}\geq \frac R{4 c_T}\r)+\mathbb{P}\le(\La_\mu(T)^{\la}\geq \frac R{4 c c_T}-1\r)}\\
\vs
\ds{\leq \mathbb{P}\le(\La_\mu(T)\geq \frac 1T\log\le(\frac R2\r)\r)+\mathbb{P}\le(\La_\mu(T)\geq \le(\frac R{4 c c_T}-1\r)^{\frac 1\la}\r)+\mathbb{P}\le(|\gamma|^2_{\H^1}+|\theta|_{L^{\la+1}}^{\la+1}\geq \frac R{4 c_T}\r).}
\end{array}\]
Due to \eqref{ghc303}, this implies that
\[\begin{array}{l}
\ds{\sup_{\mu \in\,(0,1]}\mathbb{P}\le(\sup_{t \in\,[0,T]}|v_\mu^\gamma(t)|_{H^1}^2\geq \frac{R^2}4\r)}\\
\vs
\ds{\leq c_T\le(T\le(\log\le(\frac R2\r)\r)^{-1}+\le(\frac R{4 c c_T}-1\r)^{-\frac 1\la}\r)+\frac {4 c_T}R\,\mathbb{E}\le(|\gamma|^2_{\H^1}+|\theta|_{L^{\la+1}}^{\la+1}\r).}
\end{array}\]

Now, since
\[\begin{array}{l}
\ds{\mathbb{P}\le(\sup_{t \in\,[0,T]}|u_\mu^\gamma(t)|_{H^\beta}\geq R\r)}\\
\vs
\ds{\leq \mathbb{P}\le(\sup_{t \in\,[0,T]}|\Pi_1\Gamma_\mu(t)|_{H^1}\geq\frac R2\r)+\mathbb{P}\le(\sup_{t \in\,[0,T]}|v_\mu^\gamma(t)|_{H^1}^2\geq \frac{R^2}4\r),}
\end{array}\]
by using again \eqref{ghc303}, the inequality above gives \eqref{ghc300}, for  any $\la \in\,(1,3]$.

\section{Uniform bounds for the moments of equation \eqref{eq1-abstr}}
\label{sec8}

 For every $\mu>0$, we denote by $N_\mu$ the Kolmogorov operator associated with equation \eqref{eq1-abstr}. If $\varphi:\mathcal{H}_1\to \reals$  is a twice continuously differentiable mapping, with
 \[\text{Tr}\le[D^2\varphi(z) Q_\mu Q^\star_\mu\r]<\infty,\ \ \ \ z \in\,\H_1,\]
  then we have
\[N_\mu \varphi (z)=\frac 12 \text{Tr}\le[D^2\varphi(z) Q_\mu Q^\star_\mu\r]+\langle A_\mu z+B_\mu(z),D\varphi(z)\rangle_{\mathcal{H}},\ \ \ \ \ z \in\,D(A_\mu),
\]
where $Q_\mu$, $A_\mu$ and $B_\mu$ are defined respectively in \eqref{qmu}, \eqref{amu} and \eqref{bmu}.

In Section \ref{sec3}, we have seen that for every $\mu>0$, equation \eqref{eq1-abstr} admits an invariant measure $\nu_\mu$ on the space $\mathcal{H}_1$. In particular, due to invariance, we have
\begin{equation}
\label{ghc200}
\int_{{\mathcal H}_1}N_\mu\varphi(z)\,d\nu_\mu(z)=0.
\end{equation}

Now, for every $\vartheta, \mu>0$ we define
\begin{equation}
\label{h1}
K_{\vartheta,\mu}(u,v):=\mu\,|u|_{H^1}^2+(\vartheta \,\mu+1/2)\,|u|_H^2+\mu^2\,|\vartheta\,u+v|_{H}^2+\mu\,\langle u,v\rangle_H-2\,\mu\int_D\mathfrak{b}(\xi,u(\xi))\,d\xi,\end{equation}
where  $\mathfrak{b}$ is the antiderivative of $b$ that satisfies \eqref{ghc140} and \eqref{ghc430}, in case $\la \in\,(1,3]$ and $\la=1$, respectively.
We have
\begin{equation}
\label{ghc268}
\begin{array}{rl}
\ds{D_u K_{\vartheta,\mu}(u,v)\cdot h=} & \ds{2\mu\langle u,h\rangle_{H^1}+(2\vartheta^2\mu^2+2\vartheta \mu+1)\langle u,h\rangle_H}\\
& \vs
& \ds{+(1+\mu(\vartheta+1))\langle \mu v,h\rangle_H-2\mu\,\langle b(u),h\rangle_H,}\\
\vs
\ds{D_v K_{\vartheta,\mu}(u,v)\cdot h=} & \ds{\mu(1+2\vartheta \mu)\langle u,h\rangle_H+2\mu\langle \mu v,h\rangle_H,}\\
\vs
\ds{D^2_v K_{\vartheta,\mu}(u,v)=}  &  \ds{2\mu^2\,I.}
\end{array}\end{equation}

\begin{Lemma}
For every $\bar{\mu}, \bar{\vartheta} >0$ there exist some $c_1, c_2>0$ such that for every $(u,v) \in\,\mathcal{H}_1$\begin{equation}
\label{h2}
c_1\,K_\mu(u,v)\leq K_{{\vartheta},\mu}(u,v)\leq c_2\le(K_\mu(u,v)+1\r),\ \ \ \ \vartheta\leq \bar{\vartheta},\ \ \mu\leq \bar{\mu}
\end{equation}
where
\begin{equation}
\label{h5}
K_\mu(u,v)=\mu\,|u|_{H^1}^2+\mu\,|u|_{L^{\la+1}}^{\la+1}+|u|_H^2+\mu^2\,|v|_H^2.\end{equation}
\end{Lemma}

\begin{proof}
We have
\[\begin{array}{l}
\ds{K_{\vartheta,\mu}(u,v)=\mu\,|u|_{H^1}^2+\le( 1/2+\vartheta \,\mu+\vartheta^2\mu^2\r)\,|u|_H^2+\mu^2\,|v|_{H}^2}\\
\vs
\ds{+(1+2\mu\vartheta)\,\langle u,\mu v\rangle_H-2\,\mu\int_D\mathfrak{b}(\xi,u(\xi))\,d\xi.}
\end{array}\]
Due to \eqref{ghc140} and \eqref{ghc430}, since both $\mu$ and $\vartheta$ remain bounded, it is immediate to see that there exists some $c_2>0$ such that
\[K_{{\vartheta}_1,\mu}(u,v)\leq c_2\le(K_\mu(u,v)+1\r),\ \ \ \ \vartheta\leq \bar{\vartheta},\ \ \mu\leq \bar{\mu}.\]
On the other hand, since 
\[\le|\langle u,\mu v\rangle_H\r|\leq \frac 38\,|u|_H^2+\frac 23\,\mu^2\,|v|_H^2,\]
by using again \eqref{ghc140} and \eqref{ghc430}, we get
\[K_{\vartheta,\mu}(u,v)\geq \mu\,|u|_{H^1}^2+\le( 1/8+\vartheta \,\mu/4+\vartheta^2\mu^2\r)\,|u|_H^2+\mu^2/3\,|v|_{H}^2+c\,\mu|u|_{L^{\la+1}}^{\la+1}.\]
This clearly implies that there exists some $c_1>0$ such that the lower bound in \eqref{h2} is satisfied. 

\end{proof}

 We first prove a uniform bound for some moments of the invariant measure $\nu_\mu$.
 
 \begin{Lemma}
\label{a-priori1}
Under Hypotheses \ref{H1} and  \ref{H3},  we have
\begin{equation}
\label{ghc100}
\sup_{\mu \in\,(0,1/2]} \int_{{\mathcal H}_1}\le[|u|_{H^1}^2+\mu\,|v|_H^2+|u|^{\la+1}_{L^{\la+1}}\r]\,d\nu_\mu(u,v)<\infty.
\end{equation}
\end{Lemma}

\begin{proof}
If $K_{\vartheta,\mu}$ is the function introduced in \eqref{h1}, due to \eqref{ghc268} we have
\[\begin{array}{l}
\ds{N_\mu K_{\vartheta,\mu}(u,v)=2\mu\langle u,v\rangle_{H^1}+(2\vartheta^2\mu^2+2\vartheta \mu+1)\langle u,v\rangle_H+(1+\mu(\vartheta+1))\langle \mu v,v\rangle_H}\\
\vs
\ds{-2\mu\,\langle b(u),v\rangle_H+(1+2\vartheta \mu)\langle u,A u-v+b(u)\rangle_H+2\langle \mu v,Au-v+b(u)\rangle_H+\text{Tr}\, Q^2}\\
\vs
\ds{ =-(1+2\vartheta\mu)|u|_{H^1}^2-\le(1-\mu(\vartheta+1)\r)\,\mu\,|v|_H^2+2\vartheta^2\mu\langle \mu v,u\rangle_H}\\
\vs
\ds{+(1+2\vartheta\mu)\langle b(u),u\rangle_H+\text{Tr}\, Q^2.}
\end{array}\]
This implies that
\begin{equation}
\label{ghc255}
\begin{array}{l}
\ds{N_\mu K_{\vartheta,\mu}(u,v)=-\vartheta \le[|u|_{H^1}^2+\mu\,|v|_H^2+|u|_{L^{\la+1}}^{\la+1}\r]+R_{\vartheta,\mu}(u,v),}
\end{array}\end{equation}
where
\[\begin{array}{l}
\ds{R_{\vartheta,\mu}(u,v)=-(1+2\vartheta\mu-\vartheta)|u|_{H^1}^2-\le(1-\mu(\vartheta+1)-\vartheta\r)\,\mu\,|v|_H^2}\\
\vs
\ds{+(1+2\vartheta\mu)\langle b(u),u\rangle_H+\vartheta \,|u|_{L^{\la+1}}^{\la+1}
+2\vartheta^2\mu\langle \mu v,u\rangle_H+\text{Tr}\, Q^2.}
\end{array}\]

\medskip

Now, in the estimate of $R_{\vartheta,\mu}(u,v)$, we distinguish the case $\la=1$ and $\la \in\,(1,3]$.
\medskip

If $\la=1$, due to \eqref{ghc421} for every $\e>0$ we can fix $c_\e>0$ such that
\[\langle b(u),u\rangle_H\leq (L_b+\e)|u|_H^2+c_\e\leq (L_b+\e)\a_1^{-1}\,|u|_{H^1}^2+c_\e.\]
Therefore, if we pick $\bar{\e}>0$ such that $(L_b+\bar{\e})\a_1^{-1}<1$, we get
\[\begin{array}{l}
\ds{R_{\vartheta,\mu}(u,v)\leq -((1-(L_b+\bar{\e})\a^{-1})(1+2\vartheta \mu)-\vartheta( \a_1^{-1}+1+\vartheta\,\mu^2\a_1^{-1}) )|u|_{H^1}^2}\\
\vs
\ds{-\le(1-\mu(\vartheta+1)-\vartheta(1+\vartheta\mu)\r)\,\mu\,|v|_H^2
+\text{Tr}\, Q^2+c\,(1+2\vartheta\mu).}
\end{array}\]
In particular, there exists $\bar{{\vartheta}}>0$ such that
\begin{equation}
\label{ghc433}
\sup_{\mu \in\,(0,1/2]}\,R_{\vartheta,\mu}(u,v)\leq c,\ \ \ \ \vartheta\leq \bar{{\vartheta}},\ \ \ \ \ (u,v) \in\,\mathcal{H}_1.\end{equation}

\medskip

On the other hand, if $\la \in\,(1,3]$, due to \eqref{sk3-bis}, it is possible to prove that 
\[\begin{array}{l}
\ds{R_{\vartheta,\mu}(u,v)\leq -(1+2\vartheta\mu-\vartheta(1+\vartheta\,\mu^2\a_1^{-1} ))|u|_{H^1}^2-\le(1-\mu(\vartheta+1)-\vartheta(1+\vartheta\mu)\r)\,\mu\,|v|_H^2}\\
\vs
\ds{-\le(c(1+2\vartheta\mu)-\vartheta \r)\,|u|_{L^{\la+1}}^{\la+1}
+\text{Tr}\, Q^2+c\,(1+2\vartheta\mu).}
\end{array}\]
Hence, also in this case, we can find $\bar{{\vartheta}}>0$ such that \eqref{ghc433} holds.

Thanks to \eqref{ghc255}, \eqref{ghc433}  yields
\begin{equation}
\label{ghc267}
N_\mu K_{\vartheta,\mu}(u,v)\leq -\vartheta \le[|u|_{H^1}^2+\mu\,|v|_H^2+|u|_{L^{\la+1}}^{\la+1}\r]+c,\ \ \ \ \ \vartheta\leq \bar{\vartheta},\ \ \ \mu\leq 1/2,\end{equation}
so that, from \eqref{ghc200} applied to the function $K_{\bar{\vartheta},\mu}$, we obtain
\begin{equation}
\label{ghc256}
\begin{array}{l}
\ds{\int_{{\mathcal H}_1}\le[|u|_{H^1}^2+\mu\,|v|_H^2+|u|_{L^{\la+1}}^{\la+1}\r]\,d\nu_\mu(u,v)\leq \frac{c}{\bar{\vartheta}},}
\end{array}\end{equation}
for every  $\mu \in\,(0,1/2]$. 

\end{proof}

\begin{Lemma}
\label{lemma1}
Under Hypotheses \ref{H1} and  \ref{H3}, there exists $c>0$ such that for every $\mu \in\,(0,1/2)$  and $\gamma=(\theta,\eta)  \in\,L^2(\Omega;\mathcal{H}_1)$, such that $\theta \in\,L^{\la+1}(\Omega;L^{\la+1}(D))$,  we have
\begin{equation}
\label{ghc1}
\begin{array}{l}
\ds{ \mu\,\mathbb{E}\,|u^\gamma_\mu(t)|_{H^1}^2+\mathbb{E}\,|u^\gamma_\mu(t)|_H^2+\mu^2\,\mathbb{E}\,|\partial_t u^\gamma_\mu(t)|_H^2+\mu\,\mathbb{E}|u^\gamma_\mu(t)|^{\la+1}_{L^{\la+1}}}\\
\vs
\ds{\leq c\,\le(\mu\,\mathbb{E}\,|\theta|_{H^1}^2+\mathbb{E}\,|\theta|_H^2+\mu^2\,\mathbb{E}\,|\eta|_H^2+\mu\,\mathbb{E}\,|\theta|_{L^{\la+1}}^{\la+1}+t+1\r).}
\end{array}
\end{equation}

\end{Lemma}
\begin{proof}
If we apply  It\^o's formula to $(u^\gamma_\mu(t),\partial_t u^\gamma_\mu(t))$ and to the function $K_{\vartheta,\mu}$ introduced in \eqref{h1}, due to \eqref{ghc268} we have
\[dK_{\vartheta,\mu}(u^\gamma_\mu(t),\partial_t u^\gamma_\mu(t))=
N_\mu K_{\vartheta,\mu}(u^\gamma_\mu(t),\partial_t u^\gamma_\mu(t))\,dt+\langle (1+2\mu\theta)\,u^\gamma_\mu(t)+2\mu\, \partial_t u^\gamma_\mu(t),Qdw(t)\rangle.\]
Due to \eqref{ghc267}, this implies
\[dK_{\vartheta,\mu}(u^\gamma_\mu(t),\partial_t u^\gamma_\mu(t))\leq c\,dt+\langle (1+2\mu\theta)\,u^\gamma_\mu(t)+2\mu\, \partial_t u^\gamma_\mu(t),Qdw(t)\rangle,\ \ \ \ \vartheta\leq \bar{\vartheta},\ \ \ \mu\leq 1/2.\]
Integrating with respect to $t$ and then taking expectation, we get
\[\mathbb{E} K_{\vartheta,\mu}(u^\gamma_\mu(t),\partial_t u^\gamma_\mu(t))\leq \mathbb{E} K_{\vartheta,\mu}(\theta,\eta)+c\,t,\ \ \ \ \vartheta\leq \bar{\vartheta},\ \ \ \mu\leq 1/2.\]
Therefore, \eqref{ghc1} follows from estimate \eqref{h2}.

\end{proof}

\section{Uniform bounds for the exponential bounds of equation \eqref{eq1-abstr}}
\label{sec9}

We prove here some exponential estimates for the moments of the invariant measure $\nu_\mu$, for the solutions $(u^\gamma_\mu,\partial_t u^\gamma_\mu)$ of equation \eqref{eq1-abstr}, with random  initial condition $\gamma$,  and for the solution $u^\theta$ of equation \eqref{eq1-par}, with random initial condition $\gamma$.

\begin{Lemma}
Under Hypotheses \ref{H1} and  \ref{H3},  there exists $\eta>0$ such that
\begin{equation}
\label{h4}
\sup_{\mu \in\,(0,1/2]} \int_{{\mathcal H}_1}\exp\le(\eta\le[\mu\,|u|_{H^1}^2+|u|_H^2+\mu^2\,|v|_H^2+\mu\,|u|^{\la+1}_{L^{\la+1}}\r]\r)\,d\nu_\mu(u,v)<\infty.
\end{equation}
\end{Lemma}

\begin{proof}
For every $\vartheta, \mu, \d>0$, we define
\begin{equation}
\label{h6}
\Phi_{\vartheta, \mu,\d}(u,v)=\exp\le(\d\,K_{\vartheta,\mu}(u,v)\r),\end{equation}
where $K_{\vartheta,\mu}$ is the function introduced in \eqref{h1}.
Clearly, we have
\[N_\mu \Phi_{\vartheta, \mu,\d}(u,v)= \d\,\Phi_{\vartheta, \mu,\d}(u,v) N_\mu K_{\vartheta,\mu}(u,v)+\frac {\d^2}{2\mu^2}\Phi_{\vartheta, \mu,\d}(u,v) |Q(D_v K_{\vartheta,\mu}(u,v))|_H^2.\]
Due to \eqref{ghc268}, we have
\[ |Q(D_v K_{\vartheta,\mu}(u,v))|_H^2=\mu^2(1+2\vartheta \mu)^2\,|Qu|_H^2+4\mu^4|Qv|_H+4\mu^3(1+2\vartheta \mu)\langle Qu,Qv\rangle_H.\]
Therefore, thanks to \eqref{ghc267}, for every  $\mu\leq 1/2$ we have
\[\begin{array}{l}
\ds{N_\mu \Phi_{\bar{\vartheta}, \mu,\d}(u,v)\leq 
-\bar{\vartheta} \d\,\Phi_{\bar{\vartheta}, \mu,\d}(u,v)\le[|u|_{H^1}^2+\mu\,|v|_H^2+|u|_{L^{\la+1}}^{\la+1}-\frac c{\bar{\vartheta}}\r]}\\
\vs
\ds{+ {2 \d^2} \Phi_{\bar{\vartheta}, \mu,\d}(u,v)\le[(1/2+\bar{\vartheta} \mu)^2\,|Qu|_H^2+\mu^2|Qv|_H+\mu(1+2\bar{\vartheta} \mu)\langle Qu,Qv\rangle_H\r] }\\
\vs
\ds{\leq -\bar{\vartheta} \d\,\Phi_{\bar{\vartheta}, \mu,\d}(u,v)\le[\frac 12|u|_{H^1}^2+\le(\frac {\a_1}2 -\frac{4\d\,\|Q\|^2}{\bar{\vartheta}} (1/2+\bar{\vartheta} \mu)^2\r)|u|_H^2 \r.}\\
\vs
\ds{\le.+\le(1-\frac{3\d\mu\,\|Q\|^2}{\bar{\vartheta}}\r) \mu\,|v|_H^2+|u|_{L^{\la+1}}^{\la+1}-\frac c{\bar{\vartheta}}\r].}
\end{array}\]
In particular, it is immediate to check that there exist $\bar{\d}>0$ and $\la_1, \la_2>0$ such that 
\begin{equation}
\label{h13}
\begin{array}{l}
\ds{N_\mu \Phi_{\bar{\vartheta}, \mu,\bar{\d}}(u,v)\leq -\la_1\,\bar{\vartheta}\,\bar{\d}\,\Phi_{\bar{\vartheta}, \mu,\bar{\d}}(u,v)\le[|u|_{H^1}^2+|u|^2_H+\mu\,|v|_H^2+|u|_{L^{\la+1}}^{\la+1}-\frac {\la_2}{\bar{\vartheta}}\r].}
\end{array}\end{equation}

Now, if we integrate both sides above with respect to the invariant measure $\nu_\mu$, according to \eqref{ghc200} we obtain
\begin{equation}
\label{h8}
\int_{\mathcal{H}_1} \Phi_{\bar{\vartheta}, \mu,\bar{\d}}(u,v)\le[|u|_{H^1}^2+|u|^2_H+\mu\,|v|_H^2+|u|_{L^{\la+1}}^{\la+1}-\frac {\la_2}{\bar{\vartheta}}\r]\,d\nu_\mu\leq 0,\ \ \ \ \mu\leq 1/2.\end{equation}
Thanks to \eqref{h2}, this yields
\[\begin{array}{l}
\ds{\frac 1{c_2}\int_{\mathcal{H}_1} \Phi_{\bar{\vartheta}, \mu,\bar{\d}}(u,v)\le[K_{\bar{\vartheta},\mu}(u,v)-\frac{c_2(\bar{\vartheta}+\la_2)}{\bar{\vartheta}}\r]\,d\nu_\mu\leq \int_{\mathcal{H}_1} \Phi_{\bar{\vartheta}, \mu,\bar{\d}}(u,v)\le[K_{\mu}(u,v)-\frac{\la_2}{\bar{\vartheta}}\r]\,d\nu_\mu}\\
\vs
\ds{\leq \int_{\mathcal{H}_1} \Phi_{\bar{\vartheta}, \mu,\bar{\d}}(u,v)\le[|u|_{H^1}^2+|u|^2_H+\mu\,|v|_H^2+|u|_{L^{\la+1}}^{\la+1}-\frac {\la_2}{\bar{\vartheta}}\r]\,d\nu_\mu\leq 0,\ \ \ \ \mu\leq 1/2.}
\end{array}\]
It is easy to prove that there exists some $\la_3>0$ such that 
\[e^{\bar{\d} r}\le(r-\frac{c_2(\bar{\vartheta}+\la_2)}{\bar{\vartheta}}\r)\geq e^{\bar{\d} r}-\la_3,\ \ \ \ \ r\geq 0,\]
so that for every $\mu\leq 1/2$
\[\int_{\mathcal{H}_1} \le[\Phi_{\bar{\vartheta}, \mu,\bar{\d}}(u,v)-\la_3\r]\,d\nu_\mu\leq \int_{\mathcal{H}_1} \Phi_{\bar{\vartheta}, \mu,\bar{\d}}(u,v)\le[K_{\bar{\vartheta},\mu}(u,v)-\frac{c_2(\bar{\vartheta}+\la_2)}{\bar{\vartheta}}\r]\,d\nu_\mu\leq 0.\]
By using again \eqref{h2}, 
this implies
\[\int_{\mathcal{H}_1} \exp\le(\bar{\d}c_1\,K_\mu(u,v)\r)\,d\nu_\mu\leq \int_{\mathcal{H}_1} \Phi_{\bar{\vartheta}, \mu,\bar{\d}}(u,v)\,\d\nu_\mu\leq \la_3,\ \ \ \ \mu\leq 1/2,\]
so that 
 \eqref{h4} follows, with $\eta=\bar{\d}\,c_1$.

\end{proof}

\begin{Lemma}
\label{lemma-c2}
Under Hypotheses \ref{H1} and  \ref{H3}, there exist $\eta>0$ and  $c>0$ such that for every $\mu \in\,(0,1/2)$, $T>0$  and $\gamma=(\theta,\eta)  \in\,L^2(\Omega;\mathcal{H}_1)$, such that $\theta \in\,L^{\la+1}(\Omega;L^{\la+1}(D))$,  we have
\begin{equation}
\label{h7}
\begin{array}{l}
\ds{\mathbb{E}\exp\le(\eta\le[ \mu\,|u^\gamma_\mu(t)|_{H^1}^2+|u^\gamma_\mu(t)|_H^2+\mu^2\,|\partial_t u^\gamma_\mu(t)|_H^2+\mu\,|u^\gamma_\mu(t)|^{\la+1}_{L^{\la+1}}\r]\r)}\\
\vs
\ds{\leq c_T\,\mathbb{E}\exp\le(\eta\le[\mu\,|\theta|_{H^1}^2+|\theta|_H^2+\mu^2\,|\eta|_H^2+\mu\,|\theta|_{L^{\la+1}}^{\la+1}\r]\r),\ \ \ t\leq T.}
\end{array}
\end{equation}

\end{Lemma}
\begin{proof}
If we apply It\^{o}'s formula to  $(u^\gamma_\mu(t),\partial_t u^\gamma_\mu(t))$ and to the function $\Phi_{\vartheta,\mu,\d}$ introduced in \eqref{h6}, due to \eqref{ghc268} we have
\[\begin{array}{l}
\ds{d\,\Phi_{\vartheta,\mu,\d}(u^\gamma_\mu(t),\partial_t u^\gamma_\mu(t))=
N_\mu \Phi_{\vartheta,\mu,\d}(u^\gamma_\mu(t),\partial_t u^\gamma_\mu(t))\,dt}\\
\vs
\ds{+\d\, \Phi_{\vartheta,\mu,\d}(u^\gamma_\mu(t),\partial_t u^\gamma_\mu(t))\langle (1+2\mu\vartheta)\,u^\gamma_\mu(t)+2\mu\, \partial_t u^\gamma_\mu(t),Qdw(t)\rangle.}
\end{array}\]
Due to \eqref{h13}, this implies that there exists $\bar{\vartheta},\ \bar{\d}>0$ and $\la_1,\,\la_2>0$ such that
\[\begin{array}{l}
\ds{d\,\Phi_{\bar{\vartheta},\mu,\bar{\d}}(u^\gamma_\mu(t),\partial_t u^\gamma_\mu(t))}\\
\vs
\ds{\leq
-\la_1\,\bar{\vartheta}\,\bar{\d}\, \Phi_{\bar{\vartheta}, \mu,\bar{\d}}(u^\gamma_\mu(t),\partial_t u^\gamma_\mu(t))\le[|u^\gamma_\mu(t)|_{H^1}^2+|u^\gamma_\mu(t)|^2_H+\mu\,|\partial_t u^\gamma_\mu(t)|_H^2+|u^\gamma_\mu(t)|_{L^{\la+1}}^{\la+1}-\frac {\la_2}{\bar{\vartheta}}\r]}\\
\vs
\ds{+\d\, \Phi_{\bar{\vartheta},\mu,\bar{\d}}(u^\gamma_\mu(t),\partial_t u^\gamma_\mu(t))\langle (1+2\mu\vartheta)\,u^\gamma_\mu(t)+2\mu\, \partial_t u^\gamma_\mu(t),Qdw(t)\rangle,\ \ \ \ \ \mu\leq 1/2.}
\end{array}\]
Therefore, by integrating first with respect to time and then by taking expectation in both sides, we get
\[\begin{array}{l}
\ds{\mathbb{E}\, \Phi_{\bar{\vartheta},\mu,\bar{\d}}(u^\gamma_\mu(t),\partial_t u^\gamma_\mu(t))\leq \mathbb{E}\, \Phi_{\bar{\vartheta},\mu,\bar{\d}}(\vartheta,\eta)+\frac{\la_1\,\la_2\,\bar{\vartheta}\,\bar{\d}}{\bar{\vartheta}}\int_0^t\mathbb{E}\,\Phi_{\bar{\vartheta},\mu,\bar{\d}}(u^\gamma_\mu(s),\partial_t u^\gamma_\mu(s))\,ds,}
\end{array}\]
and this implies 
\[\mathbb{E}\, \Phi_{\bar{\vartheta},\mu,\bar{\d}}(u^\gamma_\mu(t),\partial_t u^\gamma_\mu(t))\leq c_T\,\mathbb{E}\, \Phi_{\bar{\vartheta},\mu,\bar{\d}}(\vartheta,\eta),\ \ \ \ t\leq T.\]
Thanks to \eqref{h5}, this implies \eqref{h7} for some $\eta>0$.

\end{proof}

By using arguments analogous but considerably simpler than those used in the proof of Lemma \ref{lemma-c2}, we can prove that the following result holds. 

\begin{Lemma}
\label{lemma-c3}
Under Hypotheses \ref{H1} and  \ref{H3}, there exist $\eta>0$ and  $c>0$ such that for every $T>0$  and $\theta \in\,L^2(\Omega;\mathcal{H}_1)$,  we have
\begin{equation}
\label{h17}
\begin{array}{l}
\ds{\mathbb{E}\exp\le(\eta\,|u^\theta(t)|_{H}^2\r)\leq c_T\,\mathbb{E}\exp\le(\eta\,|\theta|_{H}^2\r),\ \ \ t\leq T.}
\end{array}
\end{equation}
\end{Lemma}

\bigskip

\appendix

\section{Asymptotic strong Feller property}
\label{appa}

We want to give a proof of \eqref{ghc410}, namely
 for every $x \in H$ and every
$\varphi \in C^1_b(H)$
\begin{equation}
  |D P_t \varphi(x)|_H \leq C \sqrt{P_t|\varphi|^2(x)}
                            + \alpha(t) \sqrt{P_t|D\varphi|^2_H(x)},
  \label{eq:ASF:Like}
\end{equation}
for some function $\a(t)$ such that  $\alpha(t) \to 0$, as $t \to \infty$.  This bound is a
time-asymptotic smoothing estimate for $P_t$ which implies the so-called asymptotic strong Feller condition introduced in 
\cite{HM06}.

We now demonstrate how  condition \eqref{ghc411} on $Q$ combined with condition \eqref{sk3} on $b$
imply \eqref{eq:ASF:Like}.  
We denote by $D_x u^x(t)h$ the derivative of $u^x(t)$ with respect to the initial condition $x$, along the direction $h \in\,H$. Moreover, we denote by $A^x(t)v$ the Malliavin derivative of $u^x(t)$ along the admissible perturbation $v$ of the Wiener path.
For any $\varphi \in C^1_b(H)$
we have
\begin{equation}
     \label{eq:Mal:IBP:ID}
\begin{array}{l}
\ds{ \langle D P_t \varphi(x) ,h\rangle_H 
  = \mathbb{E} \langle D \varphi(u^{x}(t)), D_x u^{x}(t)h\rangle_H}\\
  \vs
  \ds{
  = \mathbb{E}\langle  D \varphi(u^{x}(t)), D_x u^{x}(t)h - A^x(t)v\rangle_H     +\mathbb{E} \langle D \varphi(u^{x}(t)) ,A^x(t)v\rangle_H         
}\\
\vs
\ds{
  = \mathbb{E}\langle  D \varphi(u^{x}(t)), D_x u^{x}(t)h - A^x(t)v\rangle_H       
     +\mathbb{E}\, \varphi(u^{x}(t)) \int_0^t \langle v(s), dw^Q(s)\rangle_H.}
     \end{array}
\end{equation}
The
last line follows from the Malliavin integration by parts formula and the
stochastic integral is interpreted in the Skorohod sense if $v$ is not
adapted.  

Now, we define $\rho(t)= \rho^{h, v}(t) := D_x u^{x}(t)h - A^x(t)v$ and we choose
\begin{equation}
 v(t)= \alpha_{\bar{n}}\, Q^{-1}P_{\bar{n}}\, \rho(t),
  \label{eq:fb:def}
\end{equation}
where $\bar{n}$ is the integer introduced in \eqref{ghc411}.
Note
that $Q$ is invertible on $H_{\bar{n}} = \mbox{span}\{e_1, \ldots, e_{\bar{n}}\}$ according to 
the assumption \eqref{ghc411} and that this choice of $v$ is adapted.  
It is immediate to check  that $\rho(t)$ satisfies
\begin{equation*}
  \partial_t \rho(t) = \Delta \rho(t) + DB(u^x(t))\rho(t) - \alpha_{\bar{n}} P_{\bar{n}}\rho(t),\ \ \ \ \rho(0)=h.
\end{equation*}
Hence, $\rho(t)$ satisfies the following estimate
\begin{equation*}
  \frac{1}{2} \frac{d}{dt} |\rho(t)|_H^2 + |\nabla \rho(t)|_{H}^2
      + \alpha_{\bar{n}} |P_{\bar{n}} \rho(t)|_H^2
      =\langle DB(u^x(t)),\rho(t)\rangle_H \leq L_b\, |\rho(t)|_H^2.
\end{equation*}
Observe that
\[\begin{array}{l}
\ds{
  |\nabla \rho(t)|_H^2 + \alpha_{\bar{n}}\, |P_{\bar{n}} \rho(t)|_H^2
  \geq   |\nabla (I-P_{\bar{n}}) \rho(t)|_H^2 + \alpha_{\bar{n}}\, |P_{\bar{n}} \rho(t)|_H^2}\\
  \vs
  \ds{ = \sum_{k ={\bar{n}}+1}^\infty \alpha_k | \langle \rho(t), e_k\rangle_H|^2
+\sum_{k =1}^{{\bar{n}}} \alpha_{\bar{n}}\, | \langle \rho(t), e_k\rangle_H|^2  \geq \alpha_{\bar{n}} |\rho(t)|_H^2,}
  \end{array}\]
so that, thanks to \eqref{sk3} we  obtain
\begin{equation*}
  \frac{d}{dt} |\rho(t)|_H^2 + (\alpha_{\bar{n}} - L_b) |\rho(t)|_H^2 \leq 0.
\end{equation*}
This implies 
\begin{equation}
|\rho(t)|_H^2  \leq |h|_H^2 e^{-(\alpha_{\bar{n}} - L_b) t},
  \label{eq:as:bnd:1}
\end{equation}
and hence for every $t\geq 0$
\begin{equation}
  \E \le|\int_0^t \langle v(s), dw^Q(s)\rangle_H\r|^2 = \int_0^t \mathbb{E}\,|Q v(s)|_H^2\,ds\leq \a_{\bar{n}}^2 |h|_H^2\int_0^t e^{(\alpha_{\bar{n}} - L_b)s}\,ds
  \leq |h|_H^2\frac{\a_{\bar{n}}^2}{\alpha_{\bar{n}} - L_b}.  \label{eq:as:bnd:2}
\end{equation}

Therefore, returning now to \eqref{eq:Mal:IBP:ID}, thanks to bounds \eqref{eq:as:bnd:1} and \eqref{eq:as:bnd:2}
we obtain
\[
\begin{array}{l}
\ds{
  |D P_t \varphi(x)|_H
  \leq \sup_{|h|_H \leq 1} 
  \mathbb{E} \,|D \varphi(u^{x}(t))|_H |\rho(t)|_H
     +\sup_{|h|_H \leq 1} \E \left|\varphi(u^{x}(t)) \int_0^t \langle v(s),dw^Q(s)\rangle_H\right|}\\
     \vs
     \ds{
 \leq \sqrt{P_t|D\varphi|^2_H(x)} 
       \sup_{|h|_H \leq 1}  \left(\E |\rho(t)|_H^2\right)^{1/2}
      +  \sqrt{P_t|\varphi|^2(x)}  
       \sup_{|h|_H \leq 1} \left(\E \left|\int_0^t \langle v(s),dw^Q(s)\rangle_H \right|^2 \right)^{1/2}}\\
       \vs
       \ds{\leq \frac{\a_{\bar{n}}}{\sqrt{\alpha_{\bar{n}} - L_b}} \sqrt{P_t|\varphi|^2(x)}  +
       e^{- (\alpha_{\bar{n}} - L_b)t} \sqrt{P_t|D\varphi|^2_H(x)},}\end{array}\]
       which is \eqref{eq:ASF:Like}.

\bigskip

\begin{Remark}
\label{RA1}
{\em \begin{enumerate}
\item Observe that the asymptotic strong Feller  condition holds without {any conditions} on $Q$ when $L_b < \alpha_1$.  In
particular we obtain the contract estimate desired in our paper even when there is no noise.  This is natural, in this 
case the $t = \infty$ dynamics contracts exponentially to zero.
\item It should be noted that the bound (\ref{eq:ASF:Like}) would be expected to hold in a much more
degenerate situation where $\bar{n}$ does not depend on how large $L_b$ is.  Instead the condition \eqref{sk3}
must be replaced with a Hormander condition, a delicate algebraic property of the interaction between the noise $w^Q$
and the nonlinear term $b$.  For brevity of presentation we may wish to omit such details and instead refer the reader
to \cite{HM06, HM, FoldesGlattHoltzRichardsThomann2013}.
\end{enumerate}}
\qed
\end{Remark}

{\bf Acknowledgments:} This work was initiated while the two authors where visiting scholars at the Mathematical Science Research Institute (MSRI) in Berkeley, in the Fall semester 2015. They both want to thank David Herzog and Jonathan Mattingly for some interesting discussions.

\end{document}